\documentclass[12pt]{amsart}
\usepackage{amscd,amssymb,graphics}
\usepackage{mathrsfs} 

\newtheorem{theorem}{Theorem}[section]
\newtheorem{corollary}[theorem]{Corollary} 

\newtheorem{proposition}[theorem]{Proposition}
\newtheorem{observation}[theorem]{Observation}
\theoremstyle{definition}
\newtheorem{definition}[theorem]{Definition}

\newtheorem{conjecture}[theorem]{Conjecture}
\theoremstyle{remark}
\newtheorem{remark}[theorem]{Remark}

\newtheorem{example}[theorem]{Example}

\numberwithin{equation}{section}

\newcommand{\abs}[1]{\lvert#1\rvert}
\def\norm#1{\left\Vert#1\right\Vert}

\def\I {{\mathbb I}}
\def\C {{\mathbb C}}
\def\N{{\mathbb N}}
\def\K{{\mathbb K}}
\def\F{{\mathbb F}}

\def\Z {{\mathbb Z}}
\def\Id{{\mathbb{Id}}}
\def\R{{\mathbb R}}

\def\tr{{\mathrm{tr}}}
\def\e{\varepsilon}

\def\Aut{{\mbox{\rm Aut}\,}}
\def\H {{\mathscr H}}
\def\Th{{\mathrm{Th}\,}}
\def\The{{\mathrm{Th}^{\exists}}}
\def\Thc{{\mathrm{Th}_{c}}}

\newcounter{quest}
\stepcounter{quest}
\def\under#1{{$\underline{\mbox{#1}}$}}

\begin{document}

\title[Hyperlinear and sofic groups]
{An introduction to hyperlinear and sofic groups}

\author[V.G. Pestov]{Vladimir G. Pestov}

\address{Department of Mathematics and Statistics, 
University of Ottawa, 585 King Edward Ave., Ottawa, Ontario, Canada K1N 6N5}

\email{vpest283@uottawa.ca}

\author[A. Kwiatkowska]{Aleksandra Kwiatkowska}

\address{Department of Mathematics,
University of Illinois at Urbana-Champaign,
1409 W. Green Street (MC-382), 
Urbana, Illinois 61801-2975, USA}

\email{akwiatk2@illinois.edu}




\maketitle

\section{Motivation: group matrix models in the sense of classical first-order logic} 

In these lectures, we will deal with a class of groups called {\em hyperlinear groups,} as well as its (possibly proper) subclass, that of {\em sofic groups.}
One natural way to get into this line of research is through the theory of operator algebras. Here, the hyperlinear groups are sometimes referred to as ``groups admitting matrix models''. This can be indeed interpreted as a genuine model-theoretic statement, within a suitable version of logic. Namely, a group $G$ is said to admit matrix models if every existential sentence of the first-order theory of $G$ is satisfied in matrix groups. 

What makes the concept interesting --- and difficult to work with --- is that at the matrix group end it is not the classical first-order logic that one has in mind, but rather a version of continuous logic with truth values in the unit interval $[0,1]$. By way of motivation, let us try to understand first what we get by considering a class of groups admitting matrix models in the sense of the traditional binary logic. 

The language of group theory, which we will denote $L$, is a first-order predicate calculus with equality, having one ternary predicate letter $S$ and a constant symbol $e$. The group operation is coded as follows: $S(x,y,z)$ if $xy=z$. In addition, we have variables, the equality symbol $=$, logical connectives, and quantifiers.
Now consider a (countable) group $G$. A formula is said to be a \emph{sentence} if all its variables are bound within quantifiers. These formulas say something definite about the structure of a group, hence the following definition.
The {\em theory} of $G$, denoted $\Th(G)$, is defined to be the set of all sentences of the predicate calculus $L$ which are valid in $G$. The fact that the language is first-order implies that the variables only range over $G$ (and not, for instance, over families of subsets of $G$). We will further denote by $\The(G)$ the subset of $\Th(G)$ consisting of all existential first-order sentences, that is, those of the form $\exists x_1\exists x_2\ldots\exists x_n\,\phi(x_1,\ldots,x_n)$.

Let us introduce an {\em ad hoc} notion. Say that $G$ {\em admits matrix models ``in the classical sense,''} if every quantifier-free (open) formula in first order logic can be satisfied in $GL(n,\K)$ for some $n$ and some field $\K$. Even more precisely: whenever $G\vDash\phi(g_1,g_2\ldots,g_k)$, where $\phi$ is an open formula and
$g_1,g_2,\ldots,g_k\in G $, then, for a suitable natural number $n$ and  some $g_1^\prime,g^\prime_2,\ldots,g^\prime_k\in GL(n,\K)$,
\[GL(n,\K)\vDash\phi(g^\prime_1,g^\prime_2,\ldots,g^\prime_k).\]
Here we take matrix groups to be as general as possible: all groups of the form $GL(n,\K)$, where $n\in\N$ and $\K$ is an arbitrary field, are allowed.

What groups admit matrix models ``in the classical sense''? It turns out this class can be described in a very transparent way, and does not in fact depend on the choice of a field $\K$. We obtain this description in the rest of the present Section. First, a preliminary observation.

\begin{observation}
\label{obs:1}
A group $G$ admits matrix models ``in the classical sense'' if and only if $G$ locally embeds into matrix groups, that is, for every finite $F\subseteq G$ there is a natural number $n$, a field $\K$, and an injective map $i\colon F\to GL(n,\K)$ so that, whenever $x,y\in F$ and $xy\in F$, one has
\[i(xy) = i(x)i(y).\]
\end{observation}

Such a mapping $i$ as above is called a {\em local monomorphism,} or a {\em partially defined monomorphism.} 

\begin{proof} 
The {\em necessity} follows from the fact that the conjunction of all the formulas of the form $\neg(g_i=g_j)$, $i\neq j$, as well as $S(g_i,g_j,g_k)$, where $g_i,g_j,g_k\in F$ and $g_ig_j=g_k$, $1\leq i,j,k\leq n$, is satisfied in $G$ and so in a suitable linear group. Here we denote $F=\{g_1,g_2,\ldots,g_n\}$ and $i(g_i)=g^\prime_i$, where $g_i^\prime$ are chosen as in the paragraph preceding the Observation.  

To prove {\em sufficiency}, let $\phi=\phi(g_1,g_2\ldots,g_k)$ be an open formula satisfied in $G$. Write $\phi$ in a disjunctive normal form. The atomic formulas are of the form $\neg(x=y)$, $S(x,y,z)$, or $\neg S(x,y,z)$. All occurences of atoms of the type $\neg S(x,y,z)$ can be replaced with formulas $S(x,y,w)\wedge \neg(w=z)$, where $w$ is a new variable suitably interpreted in $G$. Denote $\phi^\prime$ the resulting open formula, having probably more variables, which is satisfied in $G$. This $\phi^\prime$ is written in a disjunctive normal form, with atomic formulas of the kind either $\neg(x=y)$ or $S(x,y,z)$. As a consequence of our assumptions, a conjunctive clause of such atoms is satisfied in some $GL(n,\K)$, and the same of course applies to the disjunction of a set of conjuctive clauses. Since $GL(n,\K)\vDash\phi^\prime$, the formula $\phi$ is satisfied in $GL(n,\K)$ as well.
\end{proof}

To take the next step, we introduce the following notion.

\begin{definition}\label{resfin}
A group $G$ is {\em residually finite} if it satisfies one of the following equivalent conditions:
\begin{enumerate}
\item for every $g\in G$, $g\neq e$ there exists a normal subgroup $N$ of finite index such that $g\notin N$,
\item for every finite subset $F\subseteq G$ there is a homomorphism $h$ from $G$ to a finite group with $h\restriction F$ being an injection,    \label{homom}
\item $G$ is a subgroup of a direct product of a family of finite groups.
\end{enumerate}
\end{definition}

\begin{proof}[Equivalence of the conditions]
(1)$\Rightarrow$(3): For each $g\in G\setminus\{e\}$, choose a normal subgroup $N_g$ of finite index not containing $g$, and let $\pi_g\colon G\to G/N_g$ denote the corresponding quotient homomorphism. One has $\pi_g(g)\neq e$.
Consequently, the diagonal product of all $\pi_g$, sending each $x\in G$ to the element $(\pi_g(x))_{g\in G\setminus\{e\}}$ of the direct product of quotient groups, is a monomorphism. 
(3)$\Rightarrow$(2): here $h$ is a projection on the product of a suitable finite subfamily of groups. (2)$\Rightarrow$(1): take $F=\{e,g\}$.
\end{proof}

\begin{example}\label{res}
\begin{enumerate}
\item Finite groups are residually finite;
\item finitely generated abelian groups are residually finite;
\item free groups are residually finite. \label{free}
\end{enumerate}
\end{example}

\begin{proof}[Proof of (\ref{free}) in Example \ref{res}]
The free group of countably many generators $F_\infty$ can be embedded into the free group on two generators $F_2$. (Namely, if $a,b$ are free generators of $F_2$, then the conjugates $b$, $aba^{-1}$, $a^2ba^{-2}$, $\ldots$ are free generators of a subgroup they generate.) 
Therefore it it enough to do the 
proof for $F_2$. First we show that 
\[ F_2 < SL(2,\mathbb{Z}),\]
where $SL(2,\mathbb{Z})$ denotes the group of  $2\times 2$ matrices of determinant equal to 1.
We will prove that 
\begin{displaymath}
A=
\left( \begin{array}{cc}
1 & 2  \\
0 & 1  \\
\end{array} \right)
\end{displaymath}
and
\begin{displaymath}
B =
\left( \begin{array}{cc}
1 & 0  \\
2 & 1  \\
\end{array} \right)
\end{displaymath}
are free generators.

Consider subspaces of $\mathbb{R}^2$, $X=\{(x,y)\colon \abs x >\abs y \}$ and 
$Y=\{(x,y)\colon \abs x < \abs y \}$.
Note that for every $n\in\mathbb{Z}$, $A^n$ maps $Y$ into $X$, and $B^n$ maps $X$ into $Y$. 

Although it is obvious that $AB\neq \Id$, another way to see this is to observe that the conjugate $A^2BA^{-1}$ maps $Y$ into $X$. 
Hence $A^2BA^{-1}\neq\Id$, and therefore $AB\neq \Id$. The argument easily generalizes to show that an arbitrary word $A^{n_1}B^{m_1}\ldots A^{n_k}B^{m_k}\neq \I$.

To finish the proof, we have to show that $SL(2, \mathbb{Z})$ is residually finite. For every prime $p$ the quotient homomorphism 
$h_p\colon SL(2, \mathbb{Z})\to SL(2, \mathbb{Z}_p)$ sends a matrix over $\Z$ to one over $\Z_p$ by taking its entries mod $p$. The family of homomorphisms $h_p$ is easily seen to separate points in $SL(2,\Z)$.
\end{proof}

Now comes a classical result. 

\begin{theorem}[Malcev]\label{Malcev}
\label{th:malcev}
Every finitely generated subgroup $G$ of the linear group $GL(n,\mathbb{K})$ is residually finite.
\end{theorem}

\begin{proof}[Sketch of a proof]
Let $A_1, A_2,\ldots, A_n$ be any finite set of matrices generating $G$. Without loss in generality, assume that the identity matrix is among them. 
Let $X$ denote the set of elements of $\mathbb{K}$ formed as follows: whenever $k$ is an entry of $A_iA_j^{-1}$, then we put $k,k^{-1},k-1$ and $(k-1)^{-1}$ into $X$ whenever they are defined. 
Let $R\subseteq \mathbb{K}$ be the ring generated by $X$. It is an integral domain, and so for every maximal ideal $I$ of $R$ the quotient ring  $R/I$ is a field. 

Consider a natural homomorphism $\phi\colon G \to GL(n,R/I)$ induced by the quotient $\mod I$.  Note that when an element 
has its inverse in $R$, it cannot be in $I$, and so by the choice of generators of $R$, none of matrices $A_iA_j^{-1}$ is equal to the
identity in $GL(n,\mathbb{R}/I)$. Thus, $\phi(A_i)\neq \phi(A_j)$, and since $A_i$ were arbitrary, we conclude that homomorphisms $\phi$ as above separate points of $G$.

It remains to notice that the field $R/I$ is finite, because a finitely generated ring that is a field is finite. This part of the proof requires most effort, and the details can be found e.g.\, in \cite{BO}, Theorem  6.4.12.
\end{proof}

Our purpose is served by the following concept, which is more general than residual finiteness.

\begin{definition}[Vershik and Gordon \cite{VG}]
A group $G$ is said to be {\em locally embeddable into finite groups} (an {\em LEF group}, for short) if for every finite subset $F\subseteq G$ there is a partially defined monomorphism $i$ of $F$ into a finite group.
\end{definition}

\begin{remark}
Every residually finite group is LEF, which is immediate from (\ref{homom}).
\end{remark}

Another source of LEF groups is given by the following notion.

\begin{definition}
A group is {\em locally finite} if every finite set is contained in a finite subgroup
(i.e. if every finitely generated subgroup is finite).
\end{definition}

\begin{example}
$S^{fin}_\infty$= the group of finitely supported bijections of $\N$ is locally finite, but not residually finite. Indeed, the only normal subgroup of $S^{fin}_\infty$ is  the group of finitely supported bijections of $\N$ of even sign.
\end{example}

The following result is folk knowledge.

\begin{theorem}
\label{th:lef}
For a group $G$ the following are equivalent:
\begin{enumerate}
\item $G$ admits matrix models ``in the classical sense,''
\item $G$ is LEF.
\end{enumerate}
\end{theorem}

\begin{proof}
(i) $\Rightarrow$ (ii): Let $F$ be a finite subset of $G$. Then there is a local monomorpism $i\colon F\to GL(n,\K)$. The group $\langle i(F)\rangle$ is
residually finite by  Theorem \ref{Malcev}, hence there is a homomorphism $j\colon \langle i(F)\rangle\to H$, 
where $H$ is a finite group, whose restriction to $i(F)$ is injective. The composition $j\circ i \colon F\to H$ is the required partial monomorphism.

(i) $\Leftarrow$ (ii): Let $F$ be a finite subset of $G$. There is a partial monomorphism $i\colon F\to H$, where $H$ is a finite group. Let $j\colon H\to S_n$ be an embedding into a finite permutation group 
 (every finite group
embeds into some $S_n$). To finish the proof we notice that $S_n$ embeds into $GL(n,\K)$ for an arbitrary field $\K$: to a permutation $\sigma$ we assign the
matrix $A=(a_{ij})_{i\leq n,j\leq n}$ by letting $a_{ij}=1$ if $\sigma(i)=j$, and $a_{ij}=0$ otherwise.
\end{proof}

Notice that the above proof, the choice of a field $\K$ does not matter.

To finish this introductory section, we will show that not every group is LEF.
Letting $N(r_1,r_2,\ldots, r_m)$ denote the normal subgroup generated by elements $r_1,r_2,\ldots, r_m$, a group $G$ is {\em finitely presented} when $G \cong F_n/N(r_1,r_2,\ldots, r_m)$, where 
$r_1,r_2,\ldots ,r_m$ is a finite collection of relators and $n\in\N$. 

\begin{proposition}\label{nolef}
Suppose $G$ is an infinite simple finitely presented group. Then $G$ is \under{not} LEF.
\end{proposition}

\begin{proof}
Represent $G$ as $F_n/N$, where, for short, $N=N(r_1,r_2,\ldots ,r_m)$. 
Let $X$ be the set of free generators of $F_n$. Denote by $d$ the word metric on $F_n$ with respect to $X$, given by
\[ d(x,y)=\min\{i\colon y=b_1b_2\ldots b_ix;\, b_1,b_2,\ldots, b_i\in X\cup X^{-1}\}. \]
Let $R$ be so large that the $R$-ball $B_R$ around identity
in $F_n$ contains the relators $r_1,r_2,\ldots, r_m$. Denote $\pi\colon F_n\to F_n/N$ the canonical homomorphism and put
$\widetilde{B}_R=\pi(B_R)$.

Suppose that $G$ is LEF. Let $j\colon \widetilde{B}_R\to H$ be an injection into a finite group $H$ preserving partial multiplication.
Define a homomorphism $h\colon F_n\to H$ by the condiion $h(x)=j\circ\pi(x)$, $x\in X$. The kernel $N'=\ker(h)$ is a proper subset of $F_n$ and a proper superset of $N$ (as $F_n/N$ is infinite). 
This contradicts the simplicity of $G$.
\end{proof}

\begin{example}
{\em Thompson's groups:}

$F$ = all orientation-preserving piecewise linear homeomorphisms of $[0,1]$ with finitely many non-smooth points which are all contained in the set of dyadic rationals, and the slopes being integer powers of $2$.

$T$ = all orientation-preserving piecewise linear homeomorphisms of $\mathbb{T}=\R/\Z$  with finitely many non-smooth points which are contained in dyadic rationals, and slopes being integer powers of $2$.

$V$ = all orientation-preserving piecewise linear bijections of $[0,1]$ (not necessarily continuous), with finitely many points of discontinuity, all contained in the set of dyadic rationals, the slopes being integer powers of $2$.

Every group $F,T,V$ is finitely presented. Moreover $T$ and $V$, and the commutator of $F$ are simple. A standard reference to Thompson's groups is the survey by Cannon, Floyd and Parry \cite{CFP}.

By Proposition \ref{nolef}, Thompson's groups are not LEF.
\end{example}

\section{\label{s:ultraprod}Algebraic ultraproducts}

An important role played by ultraproducts in logic and model theory is well known. Hyperlinear/sofic groups are no exception, and in the subsequent sections ultraproducts of metric groups will have a significant impact. In this section we will discuss algebraic ultraproducts of groups, and show how to reformulate in their language the existence of matrix models ``in the classical sense''.  In particular, the ultraproduct technique allows for a simpler proof of Theorem \ref{th:lef}, bypassing Malcev's theorem \ref{th:malcev}.

Recall that, given a family $G_\alpha$, $\alpha\in A$,s of groups and an ultrafilter $\mathcal U$ on the index set $A$, the ({\em algebraic}) {\em ultraproduct} of the family $(G_\alpha)$ is defined as follows:
\[\left(\prod_{\alpha\in A}G_{\alpha}\right)_{\mathcal U}
= \left(\prod_{\alpha\in A}G_{\alpha}\right)/N_{\mathcal U},\]
where 
\[N_{\mathcal U} = \{x\colon x\sim_{\mathcal U} e\}\]
and
\[x\sim_{\mathcal U} y \iff \{\alpha\in A\colon x_{\alpha}=y_{\alpha}\}\in {\mathcal U}.\]
Notice that $N_{\mathcal U}$ is a normal subgroup of the direct product of groups $G_{\alpha}$.

In a similar way, one can define an algebraic ultraproduct of a family of any algebraic structures of the same signature. In particular, if $\K_\alpha$, $\alpha\in A$ are fields, then the subset ${\mathcal I}_{\mathcal U}=\{x\colon x\sim_{\mathcal U} 0\}$ is a maximal ideal of the direct product ring $\prod_{\alpha\in A}\K_{\alpha}$, and the corresponding quotient field $\left(\prod_{\alpha\in A}\K_{\alpha}\right)_{\mathcal U}$ is called the ultraproduct of the fields $\K_\alpha$ modulo $\mathcal U$. 
(It is useful to notice that the underlying set of the algebraic ultraproduct is independent of the algebraic structure, because only $=$ is used in the definition of the equivalence relation $\sim_{\mathcal U}$.)

Now it is easy to make the following observations, going back to Jerzy \L o\'s \cite{los}.

\begin{enumerate}
\item
\label{obs(1)}
Let $\mathcal U$ be a nonprincipal ultrafilter on the natural numbers, and let $A = \cup_{n}A_n$ be the union of an increasing chain of some algebraic structures (e.g.\, groups, fields, ...). Then $A$ canonically embeds in the ultraproduct $(\prod_nA_n)_{\mathcal U}$. (To every $a\in A$ one associates an equivalence class containing any eventually constant sequence stabilizing at $a$.)
\item  
\label{obs(2)}
Let $n\in\N$ and let $\K_\alpha$, $\alpha\in A$ be fields. Then for every ultrafilter $\mathcal U$ on $A$ the groups $\left(\prod_{\alpha\in A}GL(n,\K_\alpha)\right)_{\mathcal U}$ and $GL\left(n, \left(\prod_{\alpha\in A}\K_{\alpha}\right)_{\mathcal U}\right)$ are isomorphic. \par
(The canonical ring isomorphism between $M_n\left(\prod\K_{\alpha}\right)$ and $\prod_{\alpha} M_n(\K_\alpha)$ factors through the relation $\sim_{\mathcal U}$ to a ring isomorphism between $\left(\prod_{\alpha\in A}M_n(\K_\alpha)\right)_{\mathcal U}$ and $M_n\left(\left(\prod_{\alpha\in A}\K_{\alpha}\right)_{\mathcal U}\right)$. The ultraproduct of the general linear groups of $K_\alpha$ sits inside the former ring as the group of all invertible elements, while the general linear group of the ultraproduct of $\K_\alpha$ is by its very definition the group of invertible elements of the latter ring.)

\item
\label{obs(3)} The ultraproduct of a family of ultraproducts is again an ultraproduct.

\item
\label{obs(4)} The ultraproduct of a family of algebraically closed fields is algebraically closed. 

\item\label{obs(5)}
Let $X_n$ be non-empty finite sets and let $\mathcal U$ be an ultrafilter on the set of natural numbers. If for every $N\in\N$ $\{n\in\N\colon \abs{X_n}<N\}\notin {\mathcal U}$, then the cardinality of the ultraproduct of $X_n$ mod $\mathcal U$ equals $\mathfrak c$. 
\item An algebraic ultraproduct of a family of LEF groups is again an LEF group.
\end{enumerate}

The following result is weaker than Malcev's theorem (of which it is a corollary thanks to observation (2) above), but is nonetheless strong enough for our purposes.

\begin{observation}
Every field $\K$ embeds, as a subfield, into a suitable ultraproduct of a family of finite fields. 
\label{obs:embeds}
\end{observation}

\begin{proof}
We will only give an argument in the case where the cardinality of $\K$ does not exceed that of the continuum, leaving an extension to the general case to the reader. Let at first $p>0$ be a positive characteristic. Select an increasing chain of finite algebraic extensions of $\F_p$ whose union is the algebraic closure, $\overline{\F_p}$, of $\F_p$ (e.g.\, $(\F_{p^k})$), and fix a free ultrafilter on $\N$. The ultraproduct modulo $\mathcal U$ of finite fields forming this chain  contains $\overline{\F_p}$ by (\ref{obs(1)}). By (\ref{obs(3)}), there is an ultraproduct of the family $(\F_p)$ containing a non-trivial ultrapower of $\overline{\F_p}$ as a subfield. This ultrapower, denote it $\K_p$, is an algebraically closed field by (\ref{obs(4)}), and its transcendence degree is $\mathfrak c$ by force of (\ref{obs(5)}). Since in a given characteristic two algebraically closed fields of the same transcendence degree are isomorphic (Steinitz' theorem),
our result now follows in the case of prime characteristic. To settle the case of characteristic zero, notice that the ultraproduct of all fields $\K_p$ modulo a nonprincipal ultrafilter over the prime numbers is an algebraically closed field of characteristic zero and of transcendence degree continuum.
\end{proof}

In the following strengthening of Theorem \ref{th:lef}, the equivalence (\ref{th:classical:1})$\iff$(\ref{th:classical:4}) is an immediate consequence of a well-known general result in logic, see Lemma 3.8 in Chapter  9 \cite{BS}, but this is not the main point here.

\begin{theorem}
For a group $G$ the following are equivalent:
\begin{enumerate}
\item \label{th:classical:1}
$G$ admits matrix models ``in the classical sense'', that is, every existential sentence from the first-order theory of $G$ is valid in some matrix group.
\item \label{th:classical:4}
$G<\left(\prod_iGL(n_i,\K_i)\right)_{\mathcal{U}}$ for some family of fields $\K_i$, natural numbers $n_i$, and an ultrafilter $\mathcal{U}$.
\item \label{th:classical:2}
$G$ is LEF.
\item \label{th:classical:5}
$G$ embeds into the algebraic ultraproduct of a family of permutation groups of finite rank.
\item \label{th:classical:3}
For every field $\K$, $G<\left(\prod_iGL(n_i,\K)\right)_{\mathcal{U}}$ for a suitably large index set and a suitable ultrafilter $\mathcal U$.
\end{enumerate}
\label{th:classical}
\end{theorem}

\begin{proof}
(\ref{th:classical:1})$\Rightarrow $(\ref{th:classical:4}): On $\mathscr{P}_{fin}(G)$, the family of all finite subsets of $G$ ordered by inclusion,
take an ultrafilter containing all upper cones, that is, the sets
\[\{\Phi\in \mathscr{P}_{fin}(G)\colon \Phi\supseteq F\},\]
where $F\in\mathscr{P}_{fin}(G)$. For each $\Phi\in\mathscr{P}_{fin}(G)$ choose a field $\K_\Phi$, a natural number $n_\Phi$, and an injection
$j_\Phi\colon \Phi\to GL(n_\Phi,\K_\Phi)$ preserving partial multiplication (Observation \ref{obs:1}). Define
\[ j\colon G\to \left(\prod_{\Phi\in\mathscr{P}_{fin}(G)} GL(n_\Phi,\K_\Phi)\right)_\mathcal{U} \]
by
\[ j(g)=\left[j_\Phi(g)\right]_\mathcal{U} \]
(when $g\notin \Phi$, $j_\Phi(g)$ denotes an arbitrary element of $GL(n_\Phi,\K)$).
This $j$ is an embedding of groups.
\par
(\ref{th:classical:4}) $\Rightarrow$ (\ref{th:classical:2}). 
By embedding every field $\K_i$ into an ultraproduct of finite fields (Obs. \ref{obs:embeds}), and using observations (\ref{obs(2)}) and (\ref{obs(3)}), we can assume without loss in generality that all $\K_i$ are finite fields. 
Let $F\subseteq G$ be finite. For every $g\in F$, pick a representative $(j(g)_i)\in \prod_iGL(n_i,\K_i)$ of the equivalence class $[j(g)]_{\mathcal U}$. For every index $i$, there is now a well-defined mapping $F\ni g\mapsto j(g)_i\in GL(n_i,\K_i)$.
The set of all indices $i$ for which $j(g)_i$ is a local monomorphism must belong to the ultrafilter and so is non-empty. Choose an index $i$ from this set and notice that the group $GL(n_i,\K_i)$ is finite.
\par
(\ref{th:classical:2}) $\Rightarrow$ (\ref{th:classical:5}): A similar argument to the proof of implication (\ref{th:classical:1})$\Rightarrow $(\ref{th:classical:4}), only take as $j_{\Phi}$ a local monomorphism from $\Phi$ into a suitable finite group of permutations (which exists since $G$ is assumed LEF).
\par
(\ref{th:classical:5}) $\Rightarrow$ (\ref{th:classical:3}): Here use the fact that $S_n$ sits inside of the group $GL(n,\K)$ as a subgroup for every $\K$ and $n$.
\par
(\ref{th:classical:3})$\Rightarrow $(\ref{th:classical:1}): Let $j\colon G<\prod_i\left(GL(n_i,\K)\right)_{\mathcal{U}}$ be an embedding, and let $F\subseteq G $ be finite.
Then $j\restriction F$ is a partial monomorphism, and so
$\{i\colon j_i\restriction F \mbox{ is a partial monomorphism } \}\in\mathcal{U}$ (so in particular is nonempty).
The result now follows by Observation \ref{obs:1}.
\end{proof}

Overall, we can see that theory of groups admitting matrix models ``in the classical sense'' is more or less fully understood. This approach can be seen as a ``toy example'' (to borrow another expression from theoretical physics) of more interesting and mysterious theories of group matrix models, to which we proceed now.

\section{Ultraproducts of metric structures}

The concept of a matrix model adequate for the needs of operator algebraists is less strict than the one ``in classical sense''. We do not aim to ascertain that two elements of a matrix group, $x$ and $y$, are equal. Instead, given an $\e>0$, we are allowed to interpret a formula $x=y$ in a matrix group in such a way that the ``truth value'' of the equality is $> 1-\e$. This is understood in the sense
\[d(x,y)<\e,\]
where $d$ is a distance on a matrix group in question and $x,y$ are elements of the group.
Accordingly, instead of the algebraic ultraproduct of matrix groups, we will form the {\em metric ultraproduct,} factoring out pairs of infinitesimally close elements.

The aim of this section is to formulate an adequate version of an ultraproduct of a family of metric groups, and to give some examples.

To make a good choice of a distance $d$ as above, let us first examine the notion of the Banach space ultraproduct, which is well established.

\subsection{Ultraproducts of normed spaces}

Let $(E_{\alpha})_{\alpha\in A}$ be a family of normed spaces and let ${\mathcal U}$ be an ultrafilter on the index set $A$.
Define the {\em $\ell^{\infty}$-type sum} of the spaces $E_{\alpha}$,
\[{\mathscr E}=\oplus^{\ell^\infty}E_\alpha=\left\{x\in\prod_{\alpha}E_\alpha\colon \sup_{\alpha}\norm{x_\alpha}<\infty\right\}.\] 

This $\mathscr E$ is a normed linear space containing  every $E_\alpha$ as a normed subspace. The norm on $\mathscr E$ is given by:
\[\norm x = \sup_{\alpha\in A}\norm{x_\alpha}_{\alpha}.\]

Consider
\[{\mathscr N_{{\mathcal U}}}=\left\{x\colon \lim_{\alpha\to{\mathcal U}}\norm{x_{\alpha}}=0\right\},\]
where we let $\lim_{\alpha\to{\mathcal U}}y_\alpha=y$ if for every $\e>0$, $\{\alpha\colon\abs{y_\alpha-y}<\e\}\in\mathcal{U}$.
If the $y_{\alpha}$ are uniformly bounded, then $\lim_{\alpha\to{\mathcal U}}y_\alpha$ exists and is unique.

This $\mathscr N_{\mathcal U}$ is a closed linear subspace of $\mathscr{E}$. Now we define the metric ultraproduct of the family $(E_{\alpha})_{\alpha\in A}$ modulo the ultrafilter ${\mathcal U}$ as the normed quotient space
\[\left(\prod E_\alpha\right)_{\mathcal U} ={\mathscr E}/{\mathscr N}_{{\mathcal U}}.\]
It is a linear space equipped with the norm
\[\norm{[x]_{\mathcal U}} =\lim_{\alpha\to{\mathcal U}}\norm{x_{\alpha}}.\]

A version of the diagonal argument shows that when the ultrafilter ${\mathcal U}$ is not countably complete (in particular, is
non-principal), then the ultraproduct $E=\left(\prod E_\alpha\right)_{\mathcal U}$ is a Banach space. To see this, let $(x_k)$ be a Cauchy sequence of elements of $E$. For each $i\in\N$, fix $N(i)$ so that 
\[\forall N^\prime,N\geq N(i),~~\norm{x_{N^\prime}-x_N}<2^{-i}.\]
For every $k$, select a representative $(x_k^\alpha)_{\alpha\in A}\in \mathscr{E}$ of the equivalence class $x_k$. Given an $i\in\N_+$, define
\[I_i=\{\alpha\in A\colon \norm{x_{N(i)}^\alpha-x_{N(i+1)}^\alpha}<2^{-i}\}.\]
Every $I_i\in {\mathcal U}$, and
without loss in generality, we may assume that $I_1\supseteq I_2\supseteq\dots$ and $\cap_{i=1}^\infty I_i=\emptyset$ (countable incompleteness of $\mathcal U$). Now define an element $x\in \mathscr{E}$ by
\[x\vert_{I_i\setminus I_{i+1}} = x_{N(i)}\vert_{I_i\setminus I_{i+1}}.\]
Then the equivalence class $[x]_{\mathcal U}$ is the limit of our Cauchy sequence $(x_k)$.
 
If for some natural number $n$ the set of indices $\alpha$ with $\dim E_{\alpha}= n$ is in ${\mathcal U}$, then the ultraproduct is a normed linear space of dimension $n$. If it is not the case for any $n$, then yet another variation of Cantor's argument establishes that the ultraproduct is a non-separable Banach space.

\subsection{Ultraproducts of metric groups}

We would like to have a similar construction for metric groups as we had for normed spaces. 
First we show that if we just equip the groups with left-invariant metrics (and every metrizable group admits a compatible left-invariant metric by the result of Kakutani below), some problems arise. Hence, we will have to assume that metrics are bi-invariant. Not every metrizable group has a compatible bi-invariant metric.
 
\begin{theorem}[Kakutani]
Every metrizable topological group admits a compatible left-invariant metric, i.e. a metric $d$ such that for every $g\in G$
\[d(gx,gy)=d(x,y).\]
\qed
\end{theorem}

Let $(G_\alpha,d_\alpha)_{\alpha\in A}$ be a family of topological groups equipped with compatible left-invariant metrics, 
and let ${\mathcal U}$ be an ultrafilter 
on $A$. 
We can form an ultraproduct of the family $(G_\alpha,d_\alpha)$ following the same steps as for normed spaces, but the resulting object will not, in general, be a metric group, only a homogeneous metric space, as the following example shows.

\begin{example}
Let $S_\infty$ denote the infinite symmetric group consisting of all self-bijections of a countably infinite set $\omega$. The {\em standard Polish topology} on $S_\infty$ is the topology of pointwise convergence on the discrete topological space $\omega$. In other words, it is 
induced from the product topology on $\omega^{\omega}$.
As shown by Kechris and Rosendal \cite{KR}, the standard Polish topology is the only non-trivial separable group topology on $S_\infty$.
This topology admits the following compatible left-invariant metric:
\[d(\sigma,\tau) =\sum_{i=1}^\infty \{2^{-i}\colon \sigma(i)\neq\tau(i)\}.\]

Let us try to form an ultrapower of the metric group $(S_{\infty},d)$ with regard to a nonprincipal ultrafilter ${\mathcal U}$ on the natural numbers. Every sequence $x\in (S_\infty)^{\N}$ is ``bounded'' in the sense that $\sup_{n}d(e_n,x_n)<\infty$, and so the analogue of the space $\mathscr E$ is the full Cartesian product group ${\mathscr G}= (S_\infty)^{\N}$ itself. Define
\[{\mathscr N}=\left\{x\colon \lim_{\alpha\to{\mathcal U}}d(x_{\alpha},e)=0\right\}.\]
The estimate 
\begin{eqnarray*}
d(xy,e)&=&d(y,x^{-1})\\ &\leq& d(y,e)+d(x^{-1},e) \\
&=&d(y,e)+d(e,x)
\end{eqnarray*}
shows that $\mathscr N$ is a subgroup of $\mathscr G$. (Notice the use of left-invariance of $d$.)
However, it is not a \emph{normal} subgroup. To see this,
consider two sequences of transpositions of $\omega$, $x=(x_i)=((i,i+1))_{i\in\omega}$ and $y=(y_i)=((1,i))$. Then it is easily seen that $x\in {\mathscr N}$, and yet $y^{-1}xy\notin {\mathscr N}$. 
Thus, although the homogeneous factor-space ${\mathscr G}/{\mathscr N}$ admits a $\mathscr G$-invariant metric
\[d(x,y)=\lim_{n\to{\mathcal U}}d_n(x_n,y_n),\]
it is not a group. 
\end{example}

If we want to get a {\em metric group} as a result of an ultraproduct construction, we must use {\em bi-invariant} metrics:
\[d(gx,gy)=d(x,y)=d(xg,yg).\]

If $(G_\alpha,d_{\alpha})$, $\alpha\in A$, is a family of groups equipped with bi-invariant metrics and ${\mathcal U}$ is an ultrafilter on the index set $A$, then the subgroup 
\[{\mathscr N}=\left\{x\colon \lim_{\alpha\to{\mathcal U}}d(x_{\alpha},e)=0\right\}\]
is easily seen to be a normal subgroup of
\begin{equation}
{\mathscr G}=\oplus^{\ell^\infty}G_\alpha=\left\{x\in\prod_{\alpha}G_\alpha\colon \sup_{\alpha}d(x_\alpha,e)<\infty\right\},
\label{eq:ellinftytypesum}
\end{equation}
and the quotient group 
\[\left(\prod_{\alpha\in A}G_{\alpha}\right)_{\mathcal U}={\mathscr G}/{\mathscr N}\]
is well-defined. 
It is a metric group equipped with the bi-invariant metric 
\[d(x{\mathscr N},y{\mathscr N})=\lim_{\alpha\to{\mathcal U}}d_{\alpha}(x_{\alpha},y_{\alpha})\]
and the corresponding group topology. It will be referred to as the {\em metric ultraproduct} of the family $(G_\alpha,d_{\alpha})_{\alpha\in A}$ modulo ${\mathcal U}$. 

Just as in the case of normed spaces, the ultraproduct of a family of groups with bi-invariant metrics is a complete topological group, which is either non-separable or locally compact (assuming ${\mathcal U}$ to be non countably complete).
Moreover, in all the examples we will be considering below, the domain of the ultraproduct coincides with the full cartesian product, because all the metrics are uniformly bounded from above. (In fact, one can always replace a bi-invariant metric $d$ on a group with the bounded bi-invariant metric  $\min\{d,1\}$). 

Here are a few of the most important examples of groups equipped with natural bi-invariant metrics.

\begin{example}
The symmetric group $S_n$ of finite rank $n$ equipped with the {\em normalized Hamming distance}:
\[d_{hamm}(\sigma,\tau)=\frac 1n\sharp\left\{i\colon\sigma(i)\neq\tau(i)\right\}.\]
\end{example}
 
\begin{example}
The unitary group of rank $n$,
\[U(n)=\{u\in M_n(\C) \colon u^\ast u=uu^\ast ={\mathrm{Id}}\},\]
equipped with the {\em normalized Hilbert-Schmidt metric}:
\[d_{HS}(u,v)=\norm{u-v}_2=\sqrt{\frac 1n\sum_{i,j=1}^n \abs{u_{ij}-v_{ij}}^2}.\]
This  is the standard $\ell^2$ distance between matrices viewed as elements of an $n^2$-dimensional 
Hermitian space $\C^{n^2}$, which is normalized so as to make the identity matrix have norm one. The metric is easily checked to be bi-invariant, by rewriting the definition of the distance,
\begin{eqnarray}
d_{HS}(u,v)&=& \frac 1{\sqrt n}\sqrt{{\mathrm{tr}\,} ((u-v)^\ast(u-v))}
\label{eq:hsdisttr}\\
&=& \sqrt{2-\widetilde{\tr}_n(u^\ast v)-\widetilde{\tr}_n(v^\ast u)},
\nonumber
\end{eqnarray}
where $\widetilde{\tr}_n=n^{-1/2}\tr$ is the normalized trace on $U(n)$,
and using the characteristic property of trace: \[{\mathrm{tr}}\,(AB)={\mathrm{tr}}\,(BA).\]
We will use the notation $U(n)_2$ for the group $U(n)$ equipped with the normalized Hilbert-Schmidt distance.
\end{example}

\begin{example}
The group $U(n)$ equipped with the {\em uniform operator metric:} 
\[d_{unif}(u,v)=\norm{u-v}=\sup_{\norm x\leq 1}\norm{(u-v)(x)}.\]
\end{example}

Larger matrix groups, such as $GL(n,\K)$ and their closed non-compact subgroups, typically do not possess any compatible bi-invariant metrics whatsoever. For instance, the following is a well-known observation. 

\begin{example}
The group of invertible matrices $GL(n,\R)$, as well as the special linear group $SL(n,\R)$, do not admit bi-invariant metrics compatible with their standard locally euclidean topology (induced from $M_n(\R)\cong\R^{n^2}$).
(Hint of a proof: if such a metric existed, then the group in question would possess {\em small invariant neighbourhoods,} that is, conjugation-invariant open sets would form a basis at identity. But this is not the case. The details can be found in \cite{HR}.)
\end{example}

\section{\label{s:hyp}
Groups admitting matrix models (hyperlinear groups)}

In this Section, we define the central concept of a hyperlinear group, and outline a version of model theory for metric structures which provides a rigorous framework for treating hyperlinear groups as groups admitting matrix models which are unitary groups with the Hilbert--Schmidt distance. 

\begin{definition}
A countable discrete group $G$ is {\em hyperlinear} (or: {\em admits matrix models}) if it is isomorphic to a subgroup of a metric ultraproduct of a suitable family of unitary groups of finite rank, with their normalized Hilbert-Schmidt distances. 

More exactly, $G$ is hyperlinear if there are a set $A$, an ultrafilter ${\mathcal U}$ on $A$, a mapping $\alpha\mapsto n({\alpha})$ and an imbedding 
\begin{equation*}
G<\left(\prod_{\alpha} (U(n(\alpha)),d_{HS})\right)_{\mathcal U}.
\end{equation*}
\end{definition}

The model theory of metric structures as developed by Ben-Yaacov, Berenstein, Ward Henson and Usvyatsov \cite{BYBHU} allows to see the above definition as a genuine statement about a possibility to interpret every sentence of the theory of $G$ in some matrix group $U(n)_2$. We will not attempt to develop this viewpoint systematically, limiting ourselves to a few indicative remarks.

The space of truth values in this version of continuous logic is the unit interval $\I=[0,1]$. The truth value is interpreted as a measure of closeness, and in particular the truth value of the formula $x=y$ is $d(x,y)$. 

The two quantifiers are $\inf$ (continuous analogue of $\exists$) and $\sup$ (analogue of $\forall$). Predicates are (bounded, uniformly continuous) functions $M^n\to [0,1]$, e.g.\, the counterpart of the equality relation $=$ is the distance function $d\colon M^2\to [0,1]$. 

Similarly, functions $M^n\to M$ are subject to the uniform continuity restriction. Connectives are all continuous functions $[0,1]^n\to [0,1]$. 
If $d$ is a trivial ($\{0,1\}$-valued) metric, 
one recovers the usual predicate logic with truth values $0$ and $1$, which have swapped their places.

Every sentence in the classical theory can be thus interpreted as a sentence formed in the continuous logic, but not vice-versa.

Formulas are defined inductively, just like in the classical logic. All variables and constants are terms, and whenever $f$ is a function symbol and $t_1,\ldots,t_n$ are terms, then $f(t_1,t_2,\ldots,t_n)$ is a term. An atomic formula is an expression of the form either $P(t_1,\ldots,t_n)$ or $d(t_1,t_2)$, where $P$ is an $n$-ary predicate symbol and $t_i$ are terms. Formulas are build from atomic formulas, using two rules: if $u\colon [0,1]^n\to [0,1]$ is a continuous function (that is, a connective) and $\varphi_1,\ldots,\varphi_n$ are formulas, then $u(\varphi_1,\ldots,\varphi_n)$ is a formula; if $\varphi$ is a formula and $x$ is a variable, then $\sup_x\varphi$ and $\inf_x\varphi$ are formulas. 

A {\em metric structure}, $\mathcal M$, is a complete bounded metric space equipped with a family of predicates and functions. For instance, if $(G,d)$ is
a complete bounded metric group and $d$ is bi-invariant, then $G$ can be treated as a metric structure equipped with the predicate $S(g,h,k)=d(gh,k)$, the inversion function $i\colon G\to G$, and the identity, given by the function $e\colon \{\ast\}\to G$ (a homomorphism from the trivial group to $G$). Notice that $S$ and $i$ are uniformly continuous due to the bi-invariance of the metric $d$.

The {\em value} of a sentence $\sigma$ in a metric structure $\mathcal M$ is a number $\sigma^{\mathcal M}\in[0,1]$ defined by induction on (variable-free) formulas, beginning with the convention that $d(t_1,t_2)^{\mathcal M}$ is just the value of the distance between $t_1$ and $t_2$. The value of $P(t_1,t_2,\ldots,t_n)$ and $u(\sigma_1,\sigma_2,\ldots,\sigma_n)$, where $P$ is an $n$-ary predicate symbol, $t_i$ are terms, $u$ is a continuous function $[0,1]^n\to [0,1]$, and $\sigma_j$ are sentences, is defined in a natural way. Finally,
\[\left(\sup_x\varphi(x)\right)^{\mathcal M}=\sup_x \varphi(x)^{\mathcal M},\]
and similarly
\[\left(\inf_x\varphi(x)\right)^{\mathcal M}=\inf_x\varphi(x)^{\mathcal M}.\]

A sentence of the form $\inf_x\varphi(x)$, where $x=(x_1,x_2,\ldots,x_n)$, is called an {\em $\inf$-sentence,} and serves as an analogue of an existential sentence in the classical binary logic. 

Notice that the normalized Hilbert-Schmidt distance $d_{HS}$ takes values in the interval $[0,2]$, so if we want the values of sentences to belong to $[0,1]$, we may wish to use the distance $d=\min\{d_{HS},1\}$ instead. 

Every formula of the first-order theory of groups admits a  ``translation'' into a formula of the continuous logic theory of groups equipped with a bi-invariant metric. Namely, the symbols $S$, $i$ and $e$ are replaced with the corresponding predicate and function symbols described above, the logical connectives $\wedge$ and $\vee$ become, respectively, continuous functions $\max$ and $\min$ from $[0,1]^2$ to $[0,1]$, while $\neg$ is replaced with the function $t\mapsto 1-t$, and the quantifiers $\exists_x$ and $\forall_x$ are turned into $\inf_x$ and $\sup_x$, accordingly. Under this translation, sentences go to sentences, existential sentences go to $\inf$-sentences, and so on. Intuitively, under this ``translation,'' exact statements become approximate.
It is in this sense that we treat sentences of $\Th(G)$ as sentences of the continuous logic theory of unitary groups $U(n)_2$ in the statement of the next result. 

In connection with iten (\ref{item:metric1a}) below, remember that $U(n)_2$ embeds isometrically into the (renormalized) Euclidean space $\ell^2(n^2)$, and so the ultraproduct of unitary groups isometrically embeds into the corresponding Hilbert space ultraproduct. In this sense, one can talk about orthogonality.

\begin{theorem}
For a group $G$, the following are equivalent.
\begin{enumerate}
\item $G$ is hyperlinear, 
\label{item:metric1}
\smallskip
\item $G$ embeds into a metric ultraproduct $\left(\prod_iU(n_i)_2\right)_{\mathcal U}$ of a family of unitary groups as an orthonormal system of vectors,
\label{item:metric1a}
\smallskip
\item For every finite $F\subseteq G$ and every $\e>0$, there are $n\in\N$ and an {\em $(F,\e)$-almost homomorphism} $j\colon F\to U(n)_2$, that is, a map with the property
\smallskip
\begin{enumerate}
\item[(a)] if $g,h\in F$ and $gh\in F$, then $d(j(g)j(h),j(gh))<\e$,
\end{enumerate}
\smallskip
which is in addition {\em uniformly injective} on $F$ in the sense that:
\smallskip
\begin{enumerate}
\item[(b)] if $g,h\in F$ and $g\neq h$, then $d(g,h)>\sqrt 2-\e$.
\end{enumerate}
\label{item:metric3}
\smallskip
\item $G$ admits matrix models in the sense of continuous logic: for every existential first-order sentence $\sigma\in \The(G)$ and each $\e>0$ there is $n$ such that
\[\sigma^{U(n)_2}<\e.\]
\label{item:metric2}
\item The same conditions (a) and (b) as in item (\ref{item:metric3}), but with $\sqrt 2-\e$ in (b) replaced by a fixed positive value, e.g.\, $10^{-10}$. 
\label{item:10-10}
\end{enumerate}
\label{th:hypcriteria}
\end{theorem}

\begin{proof}
(\ref{item:metric1})$\Rightarrow$(\ref{item:metric1a}): This is the key to the entire result, whence the rest follows easily. Given a monomorphism 
\begin{equation}
i\colon G\hookrightarrow \left(\prod_\alpha U(n_\alpha)_2\right)_{\mathcal U},
\label{eq:i}
\end{equation}
and any two distinct elements $g,h\in G$, we can of course guarantee that the images $i(g)$ and $i(h)$ are at a strictly positive distance from each other, but no more than that: something like $d(i(g),i(h))=10^{-10}$ is definitely a possibility, and in fact the image $i(G)$ in the induced topology may even happen to be a non-discrete group.
We will now construct a re-embedding, $j$, of $G$ into another metric ultraproduct of unitary groups, where the distance between $j(g)$ and $j(h)$ will be always equal to $\sqrt 2$.

The Hermitian space $M_n(\C)$, which we identify with $\C^{n^2}$, admits a natural action of the unitary group $U(n)$ by conjugations:
\[u\cdot M = u^{\ast}Mu,~~u\in U(n),~~M\in M_n(\C).\]
This action is by linear operators and preserves the Hermitian inner product, for instance, since Formula (\ref{eq:hsdisttr}), without the scalar factor in front, gives the Hilbert distance on the space $M_n(\C)$. Thus, we obtain a unitary representation $U(n)\to U(n^2)$. Denote it $i_n^{(2)}$ (in fact, the more precise symbol would be $\bar i_n\otimes i_n$).

We want to compute the distance induced on $U(n)$ by the embedding $i_n^{(2)}\colon U(n)\hookrightarrow U(n^2)_2$ as above. For this purpose, again according to Equation (\ref{eq:hsdisttr}), it suffices to know the restriction of the trace $\tr_{n^2}$ to $U(n)$, that is, the composition $\tr_{n^2}\circ i_n^{(2)}$. The matrices $E_{ij}$ whose $(i,j)$-th position is one and the rest are zeros form an orthonormal basis of $M_n(\C)$, and so for every linear operator $T$ on $M_n(\C)$,
\[\tr(T)=\sum_{ij} \langle T(E_{ij}),E_{ij}\rangle = \sum_{ij}\left(T(E_{ij})\right)_{ij}.\]
Since $(u\cdot E_{ij})_{ij}=(u^\ast E_{ij}u)_{ij}=\overline{u_{ji}}u_{ji}$, we conclude: for every $u\in U(n)$,
\[\tr_{n^2}(i_n^{(2)}(u)) = \overline{\tr_n(u)}\tr_n(u)=\abs{\tr_n(u)}^2.\]
The same clearly holds with regard to the normalized traces on both unitary groups. Since at the same time the trace is a linear functional, we deduce from Equation (\ref{eq:hsdisttr}):
\begin{eqnarray*}
d_{HS,n^2}\left(i_n^{(2)}(u),i_n^{(2)}(v)\right)&=&
\frac{1}{\sqrt{n^2}}\sqrt{\tr_{n^2}((i_n^{(2)}(u)-i_n^{(2)}(v))^{\ast}
(i_n^{(2)}(u)-i_n^{(2)}(v)))}\\
&=&
\frac{1}{n}\sqrt{\tr_{n^2}(2{\mathbb I}-\tr_{n^2}(i_n^{(2)}(u^\ast v))- \tr_{n^2}(i_n^{(2)}(v^\ast u))) }\\
&=& \sqrt{2-2\left\vert\widetilde{\tr}_n(u^\ast v)\right\vert^2},
\end{eqnarray*}
where $\widetilde{\tr}_n$ denotes the normalized trace on $U(n)$.

Compare this to:
\[d_{HS,n}(u,v)= \sqrt{2-\widetilde{\tr}_n(u^\ast v)-\widetilde{\tr}_n(v^\ast u)}.\]
Since $\widetilde{\tr}_n(u^\ast v)+\widetilde{\tr}_n(v^\ast u) = 2\left\vert \widetilde{\tr}_n(u^\ast v)\right\vert$, the last two equations imply:
\begin{equation}
d_{HS,n^2}\left(i_n^{(2)}(u),i_n^{(2)}(v)\right)
= d_{HS,n}(u,v)\sqrt{2-\frac{d_{HS,n}(u,v)^2}2}.
\end{equation}

If we now define recurrently group embeddings $i_n^{(2^k)}\colon U(n)\hookrightarrow U(n^{2^k})$, $k=2,3,\ldots$, it follows that for any two elements $u,v\in U(n)$ satisfying $0<d_{HS,n}(u,v)< 2$ the iterated distances inside of the groups $U(n^{2^k})_2$ converge to $\sqrt 2$ in the limit $k\to\infty$.  

Now let us get back to the initial group embedding from Equation (\ref{eq:i}). First, we want to assure that the pairwise distances within the image $i(G)$ are strictly less than $2$. This is achieved by throwing inside the ultraproduct a pile of rubbish, as follows: embed every $U(n_\alpha)$ into a unitary group of twice the rank using block-diagonal matrices:
\[U(n_\alpha)\ni u \mapsto\left(\begin{array}{c|c}
u & 0 \\ \hline 
0 & {\mathbb I}_n
\end{array}\right)\in U(2n_\alpha).\]
The resulting composition mapping 
\[i^{\prime}\colon G\to  \left(\prod_\alpha U(n_\alpha)_2\right)_{\mathcal U}\to \left(\prod_\alpha U(2n_\alpha)_2\right)_{\mathcal U}\]
is still a group monomorphism, but all the distances between elements of $G$ are now cut by half and so the diameter of $i^\prime(G)$ is $\leq 1$. So we can assume without loss in generality that the original embedding $i$ has this property.

On the new index set $B=A\times\N_+$ choose an ultrafilter $\mathcal V$ satisfying two properties:
\begin{enumerate}
\item 
the projection of $\mathcal V$ along the first coordinate is the initial ultrafilter $\mathcal U$, and
\item if the intersection of a subset $X\subseteq A\times\N_+$ with every fiber $\{a\}\times\N_+$ is cofinite, then $X\in {\mathcal V}$.
\end{enumerate}
Lift the monomorphism $i$ in an arbitrary way to a map $\bar i\colon G\to \prod_{\alpha\in A}U(n_{\alpha})$ and define a map $\bar j\colon G\to \prod_{(\alpha,k)\in B}U\left(n_{\alpha}^{2^k}\right)$ by letting
\[\bar j_{\alpha,k}(g) = i_{n_\alpha}^{(2^k)}(i_{\alpha}(g)).\]
This $\bar j$ determines a map $j\colon G\to\left(\prod_{(\alpha,k)\in B}U\left(n_{\alpha}^{2^k}\right)\right)_{\mathcal V}$, and it is not hard to see that $j$ is a group monomorphism with the property that the images of every two distinct elements of $G$ are at a distance exactly $\sqrt 2$ from each other.
\smallskip

(\ref{item:metric1a})$\Rightarrow$(\ref{item:metric3}): Given a group embedding $i$ as in Equation (\ref{eq:i}) with the property that $i(g)$ and $i(h)$ are orthogonal whenever $g\neq h$, let $F\subseteq G$ be finite and let $\e>0$. Let $\bar i$ denote any lifting of $i$ to a map from $G$ to the direct product of $U(n_\alpha)$. Denote $C$ the set of indices $\alpha$ for which $\bar i_\alpha$ is an $(F,\e)$-almost monomorphism which in addition satisfies 
\[\sqrt 2-\e<d(\bar i_\alpha(g),\bar i_{\alpha}(h))<\sqrt 2 +\e\]
for all $g,h\in F$, $g\neq h$. Then $C\in {\mathcal U}$ and in particular $C$ is non-empty.
\smallskip

(\ref{item:metric3})$\Rightarrow$(\ref{item:metric2}): again, as in the proof of sufficiency in Observation \ref{obs:1}, it is enough to consider the case of a conjunction of atomic formulas of the form $\neg(x=y)$ or $S(x,y,z)$. When dealing with negation, remember that we replace the metric $d_{HS}$ with $\min\{1,d_{HS}\}$, and so the condition $d_{HS}(x,y)>\sqrt 2-\e$ implies
$\neg(x=y)^{U(n)_2}<\e$. 

(\ref{item:metric2})$\Rightarrow$(\ref{item:10-10}): quite obvious.

(\ref{item:10-10})$\Rightarrow$(\ref{item:metric1}): the argument is just a slight variation of the proof of the implication
``(\ref{th:classical:1})$\Rightarrow $(\ref{th:classical:4})'' in Theorem \ref{th:classical}, so we leave the details to the reader. They can be found in the proof of Th. 3.5 in \cite{pestov}, see also Corollary 5.10 in \cite{BYBHU}.
\end{proof}

The meat of the above theorem (the equivalence of conditions (1,2,3,5)) is variously attributed either to Radulescu \cite{radulescu} or to an earlier work of Kirchberg.

A {\em closed $L$-condition} is an expression of the form $\varphi^{\mathcal M}=0$, where $\varphi$ is a sentence of the language of continuous logic. Notice that this means ${\mathcal M} \vDash\varphi$.
A {\em theory} is a set of closed $L$-conditions. It may be slightly unsettling to observe that the characterization of hyperlinear groups in Theorem \ref{th:hypcriteria}, item (\ref{item:metric2}), is not, strictly speaking, stated in terms of the {\em theory of unitary groups}. However, this is most naturally fixed, as follows. In the statement of the following result, the {\em ultrapower} of $G$, as usual, means the ultraproduct of a family of metric groups metrically isomorphic to $G$.

\begin{corollary}
Let $U=(U,d)$ be a group equipped with a bi-invariant metric and satisfying two conditions:
\begin{enumerate}
\item[(a)] For every $n$, the group $U(n)_2$ embeds into $U$ as a metric subgroup, and
\item[(b)] $U$ embeds into an ultraproduct of groups $U(n)_2$ as a metric subgroup.
\end{enumerate}
Then the following are equivalent for an arbitrary group $G$:
\begin{enumerate}
\item 
\label{tc:1}
$G$ is hyperlinear,
\item 
\label{tc:2}
every existential sentence $\sigma$ of the first-order theory of $G$ satisfies $\sigma^{U}=0$, that is, belongs to the continuous theory of $U$:
\[\The(G)\subseteq \Thc(U).\]
\item
\label{tc:3}
$G$ embeds into a metric ultrapower of $U$.
\end{enumerate}
\label{c:u}
\end{corollary}

(Here $\Thc(U)$ denotes of course the first-order continuous logic theory of $U$.) 

\begin{proof}
(\ref{tc:1})$\Rightarrow$(\ref{tc:2}): Thanks to Theorem \ref{th:hypcriteria}, for every $\e>0$ we have $\sigma^{U}<\e$, whence the conclusion follows.
\smallskip

(\ref{tc:2})$\Rightarrow$(\ref{tc:1}): Denote, for simplicity, by $P$ a metric ultraproduct of the unitary groups of finite rank containing $U$ as a metric subgroup.
If $\sigma\in \The(G)$ and $\e>0$ is any, then we have $\sigma^{P}<\e$ and a by now standard argument using a lift of the monomorphism $G\hookrightarrow U\hookrightarrow P$ to the direct product of unitary groups implies the existence of $n$ with $\sigma^{U(n)_2}<\e$. Now we conclude by Theorem \ref{th:hypcriteria}.
\smallskip

(\ref{tc:1})$\iff$(\ref{tc:3}): It is enough to notice that every metric ultrapower of $U$ is contained in some metric ultraproduct of unitary groups of finite rank, via a rather straightforward reindexing procedure, and vice versa.
\end{proof}

Here is just one example of a group $U$ as above, and the most economical one.

\begin{example}
The group monomorphism
\[U(n)_2\ni u\mapsto  \left( \begin{array}{cc}
u & 0 \\
0& u
\end{array}\right)\in U(2n)_2\]
is an isometry with regard to the normalized Hilbert-Schmidt distances on both groups. It generates an increasing chain of unitary groups
\[U(1)_2<U(2)_2<\ldots<U(2^n)_2<U(2^{n+1})_2<\ldots.\]
The union of the chain,  $\cup_{n=1}^{\infty} U(2^n)$, is a group which supports a naturally defined bi-invariant Hilbert-Schmidt metric.
The completion of this group is a Polish group, denoted 
$U(R)$ and called, in full, the ``unitary group of the hyperfinite factor $R$ of type $II_1$ equipped with the ultraweak topology.'' Regarded as a metric group, $U(R)$ clearly satisfies the hypothesis of Corollary \ref{c:u}. 
\end{example}

It remains unknown whether every group is hyperlinear, and this is presently one of the main open questions of the theory.

The origin of the concept of a hyperlinear group is described in the survey \cite{pestov}, \S 7, whose duplication we try to avoid inasmuch as possible. In brief, it is motivated by Connes' Embedding Conjecture \cite{connes-injective}, which states that every von Neumann factor of type $II_1$ embeds into an ultrapower of $R$, the (unique) hyperfinite factor of type $II_1$, traditionally denoted $R^\omega$. Existence of a non-hyperlinear group would imply a negative answer to Connes' Embedding Conjecture, and send far-reaching ripples. 

In a highly interesting historical remark at the beginning of a recent preprint \cite{sherman}, David Sherman brings attention to the 1954 article \cite{wright} by Fred Wright, which essentially contained a construction of the ultraproduct of von Neumann factors of type $II_1$ (the so-called {\em tracial ultraproduct}). It was done in the language of maximal ideals rather than ultrafilters, but the two approaches are equivalent. Sherman notes: ``An amusing consequence is that the tracial ultraproduct is older than the ``classical'' ultraproduct from model theory 
(\L o\'s \cite{los} in 1955).'' Since the metric ultraproduct of unitary groups $\left(\prod U(n)_2\right)_{\mathcal U}$ is isomorphic, as a metric subgroup, to the unitary group of the tracial ultraproduct of finite-dimensional matrix algebras (considered by Wright as an example, {\em loco citato}), it means that the metric ultraproduct of groups considered in these notes historically made its appearance --- albeit an implicit one --- before the algebraic ultraproduct of groups, as described in Section \ref{s:ultraprod}.

\section{Sofic groups}

Sofic groups are those groups admitting models which are finite symmetric groups with the normalized Hamming distance --- that is, matrix models of a more restrictive kind, meaning that every sofic group is hyperlinear. This Section largely mirrors the preceding Section \ref{s:hyp}, and we show first examples of sofic groups towards the end.

\begin{definition}\label{defsof}
A discrete group $G$ is {\em sofic} if it is isomorphic to a subgroup of a metric ultraproduct of a suitable family of symmetric groups of finite rank with their normalized Hamming distances. 

In other words, there is a set $A$, a nonprincipal ultrafilter ${\mathcal U}$ on $A$, and a mapping $\alpha\mapsto n({\alpha})$ so that
\begin{equation*}
\label{eq:sofic}
G<\left(\prod_{\alpha} (S_{n({\alpha})},d_{hamm})\right)_{\mathcal U}.
\end{equation*}
\end{definition}

Again, one can reformulate the concept in the language of the existence of models which are finite symmetric groups with their normalized Hamming distances.

\begin{theorem}
For a group $G$, the following conditions are equivalent.
\begin{enumerate}
\item 
\label{item:sofic1}
$G$ is sofic.
\item 
\label{item:sofic2}
$G$ embeds into an ultraproduct of symmetric groups of finite rank in such a way that every two distinct elements in the image are at a distance $1$ from each other.
\item 
\label{item:sofic3}
For every finite $F\subseteq G$ and every $\e>0$, there are $n$ and an $(F,\e)$-almost homomorphism $j\colon F\to S_n$ which is uniformly injective: $d_{hamm}(j(g),j(h))\geq 10^{-10}$ whenever $g,h\in F$ and $g\neq h$. 
\item For every existential sentence $\sigma$ of the first-order theory of $G$ and each $\e>0$, there exists $n$ so that
\[\sigma^{S_n}<\e.\]
\end{enumerate}
\label{th:soficcriterion}
\end{theorem}

The proof is very similar to, but quite a bit easier than, that of Theorem \ref{th:hypcriteria}, with the implication (\ref{item:sofic1})$\Rightarrow$(\ref{item:sofic2}) again being central. We have chosen not to duplicate the proof which can be found in the survey \cite{pestov} of the first-named author (see Theorem 3.5). Theorem \ref{th:soficcriterion} (save condition (4)) was established by Elek and Szab\'o \cite{ES}, who were, it seems, already aware of Radulescu's result for hyperlinear groups \cite{radulescu} (that is, our Theorem \ref{th:hypcriteria}).

Every finite symmetric group $S_n$ canonically embeds into the unitary group $U(n)$, and their ditances are easily seen to satisfy:
\[d_{hamm}(\sigma,\tau) = \frac 12 \left(d_{HS}(A_\sigma,A_\tau)\right)^2.\]
Now Condition (\ref{item:sofic3}) in Theorem \ref{th:soficcriterion}, jointly with Condition (\ref{item:metric3}) in Theorem \ref{th:hypcriteria}, imply:

\begin{corollary}[Elek and Szab\'o \cite{ES}] 
Every sofic group is hyperlinear. \qed
\end{corollary}

Again, it is unknown whether every group is sofic, or whether every hyperlinear group is sofic. 

One can also state a close analogue of Corollary \ref{c:u}.

\begin{corollary}
Let $S=(S,d)$ be a group equipped with a bi-invariant metric and satisfying two conditions:
\begin{enumerate}
\item[(a)] For every $n$, the group $S_n$, with its normalized Hamming distance, embeds into $S$ as a metric subgroup, and
\item[(b)] $S$ embeds into an ultraproduct of groups $S_n$ as a metric subgroup.
\end{enumerate}
Then the following are equivalent for an arbitrary group $G$:
\begin{enumerate}
\item 
\label{tcs:1}
$G$ is sofic,
\item 
\label{tcs:2}
every existential sentence $\sigma$ of the first-order theory of $G$ satisfies $\sigma^{S}=0$:
\[\The(G)\subseteq \Thc(S).\]
\item
\label{tcs:3}
$G$ embeds into a metric ultrapower of $S$.
\end{enumerate}
\label{c:s}
\end{corollary}

There exist natural examples of groups $S$ satisfying the above, and the following is, in a sense, the simplest among them.

\begin{example}
Let $\lambda$ denote the Lebesgue measure on the unit interval $[0,1]$.
Equip the group $\Aut([0,1],\lambda)$ of measure-preserving transformations with the uniform metric
\[d_{unif}(\sigma,\tau)=\lambda\{t\in [0,1]\colon \sigma(t)\neq\tau(t)\}.\]
This metric is bi-invariant, complete, and makes $\Aut([0,1],\lambda)$ into a non-separable group. 

For every $n$, realize $S_{2^n}$ as the group of measure preserving transformations of the interval $[0,1]$ whose restriction to every interval $[i2^{-n},(i+1)2^{-n}]$, $i=0,1,\ldots,2^{n-1}$, is a translation. The restriction of the uniform distance $d_{unif}$ to $S_{2^n}$ equals the normalized Hamming distance, and for every $n$
\[S_{2^n}<S_{2^{n+1}}.\]
The uniform closure of the union of the chain of subgroups $\cup_{n}S_{2^n}$ in the group $\Aut([0,1],\lambda)$ is denoted $[E_0]$. This is a Polish group equipped with a bi-invariant metric. The name for this object is somewhat long: {\em ``the full group of the hyperfinite aperiodic ergodic measure-preserving equivalence relation''.} 
It is easy to verify that the group $[E_0]$ satisfies the assumptions of Corollary \ref{c:s}. In addition, it sits naturally as a closed topological subgroup of the group $U(R)$. 

Here are just a few words of explanation of where this group and its name come from; we refer to \cite{KM} for details and references.
Let $\mathscr R$ be a Borel equivalence relation on a standard Borel space $X$ equipped with a finite measure $\mu$. 
The {\em full group} of $\mathscr R$ in the sense of Dye, denoted $[{\mathscr R}]$, is the subgroup of all non-singular transformations $\sigma$ of $(X,\mu)$ with the property $(x,\sigma(x))\in {\mathscr R}$ for $\mu$-a.e. $x$. (A transformation is non-singular if it takes null sets to null sets.)
If equipped with the uniform metric, $[{\mathscr R}]$ is a Polish group. 

The relation $\mathscr R$ is {\em hyperfinite} if it is the union of an increasing chain of relations each having finite equivalence classes, and it is {\em aperiodic} if $\mu$-almost all $\mathscr R$-equivalence classes are infinite. The relation $\mathscr R$ is ergodic if every $\mathscr R$-saturated measurable subset of $X$ is either a null set or has full measure. Finally, $\mathscr R$ is {\em measure-preserving} if the full group $[{\mathscr R}]$ consists of measure-preserving transformations.
An example of such a relation is the {\em tail equivalence relation} on the compact space $\{0,1\}^\omega$ equipped with the product of uniform measures: $xE_0y$ if and only if there is $N$ such that for all $n\geq N$, $x_n=y_n$. 
It can be proved that the group $[{\mathscr R}]$ of every hyperfinite ergodic ergodic measure-preserving equivalence relation is isometrically isomorphic to $[E_0]$ as defined above. 
\end{example}

Historically the first ever example of a hyperlinear group which is not obviously such belongs to Connes \cite{connes-injective} and, independently, Simon Wassermann \cite{wassermann}: the free non-abelian group. 
Recall that every non-abelian free group is LEF.

\begin{theorem}
Every LEF group is sofic (hence hyperlinear).
\label{th:lefsof}
\end{theorem}

\begin{proof}
It is easily seen that the LEF property of a group $G$ is equivalent to the following: for every sentence $\sigma\in \The(G)$ there is $n$ with $\sigma^{S_n}=0$. Now Condition (\ref{item:sofic3}) of Theorem \ref{th:soficcriterion} applies.
\end{proof}

In fact, it appears that all the presently known particular examples of hyperlinear groups are at the same time known to be sofic. However, there is an interesting class of groups potentially able to distinguish between soficity and hyperlinearity and pointed out by Ozawa in \cite{ozawa}. These are wreath products $\Z_2\wr G$ where $G$ is a sofic group, that is, the semi-direct products $G\ltimes \Z_2^G$ with regard to the natural action of $G$ by permutations. That such groups are hyperlinear, follows from Theorem 2 in Elek and Lippner \cite{EL}.

The second basic class of sofic groups is that of {\em amenable} groups, and we proceed to examine it in the next section.

\section{Amenability}

Amenability has its origins in the Banach--Tarski paradox which says that we can partition a solid unit ball in the three-dimensional Euclidean space $\mathbb{R}^3$ into a finite number of pieces
(five is enough) such that by rearranging them via isometries of $\mathbb{R}^3$, we can obtain two unit solid balls in 
$\mathbb{R}^3$.
 
Expanding on this idea, we say that a group $G$ admits a {\em paradoxical decomposition} if there are pairwise disjoint subsets $A_1,A_2,\ldots, A_n$, $B_1,B_2,\ldots, B_m$ 
of $G$ and elements $g_1,g_2,\ldots,g_n,h_1,h_2,\ldots,h_m\in G$ such that 
 \[G=\bigcup_ig_iA_i=\bigcup_jh_jB_j.\]
 
 \begin{example}
The group $F_2$ admits a paradoxical decomposition.
Indeed, let $a,b$ be free generators of $F_2$. Let $w(a)$  be the set of all reduced words in $F_2$ with the first letter equal to $a$.
 Similarly define $w(a^{-1}), w(b)$ and $w(b^{-1})$.

Note that
 \begin{eqnarray*}
 F_2 & = & \{e\}\cup w(a)\cup w(a^{-1}) \cup w(b) \cup w(b^{-1})\\
  & = & w(a) \cup a\left(w(a^{-1})\right) \\
  & = & w(b) \cup b\left(w(b^{-1})\right).\\
  \end{eqnarray*}
\end{example}

\begin{theorem}
\label{amena}
For a discrete group $G$ the following are equivalent:
\begin{enumerate}
\item $G$ does \under{not} admit a paradoxical decomposition,
\item $G$ admits a finitely additive probability measure $\mu$, defined on the power set ${\mathcal P}(G)$, which is invariant under left translations,
\item $G$ admits a {\em (left)} {\em invariant mean,} i.e. a positive linear functional  
$\phi\colon \ell^\infty(G)\to \C$ satisfying $\phi(1)=1$ and invariant under left translations,
\item There is an invariant regular probability measure on the Stone-\v Cech compactification $\beta G$,
\item 
\label{item:folner}
{\em (F\o lner's condition)}: for every
finite $F\subseteq G$ and $\e>0$, there is a finite $\Phi\subseteq G$ (a {\em F\o lner set} for $F$ and $\e$) such that for each $g\in F$,
\[\vert g\Phi\bigtriangleup\Phi\vert< \e \vert\Phi\vert,\]
\item {\em (Reiter's condition (P1))}:
for every finite $F\subseteq G$ and $\e>0$, there is $f\in\ell^1(G)$ with 
$\norm{ f}_{\ell^1(G)}=1$ and such that for each $g\in F$, 
$\norm{f-\,\,^gf}_{\ell^1}<\e$. {\em (Here $^gf(x)=f(g^{-1}x)$).}
\end{enumerate}
\end{theorem}

A countable group is called {\em amenable} if it satisfies one of the equivalent  conditions of Theorem \ref{amena}. 

\begin{proof}
(2)$\Rightarrow$(1): It is clear that the presence of finitely additive measure invariant under left translations precludes the possibility of a paradoxical decomposition.

(3)$\Rightarrow$(2): Put $\mu(A)=\phi(\chi_A)$.

(3)$\Leftrightarrow$(4): Banach algebras $\ell^\infty(G)$ and $C(\beta G)$ are canonically isomorphic, and the isomorphism preserves the action of $G$ by isometries. Positive linear functionals on $C(\beta G)$ 
correspond to regular measures on $\beta G$ via Riesz representation theorem.
Left-invariant means on $\ell^\infty(G)$ correspond to invariant regular probability measures on $C(\beta G)$.

(5)$\Rightarrow$(6): For a given $\e>0$ and $F\subseteq G$ take $\Phi$ as in (\ref{item:folner}). Now the function
\[f=\frac{\chi_\Phi}{\vert\Phi\vert}\in\ell^1(G)\] 
has the required property.

(6)$\Rightarrow$(3): 
The closed unit ball $B$ of the dual Banach space to $\ell^{\infty}(G)$ is compact in the weakest topology making all evaluation mappings $\phi\mapsto \phi(x)$ continuous, $x\in\ell^{\infty}(G)$ (the {\em Banach--Alaoglu theorem}).
The set of all means on $\ell^{\infty}(G)$ (that is, positive linear functionals $\phi$ satisfying $\phi(1)=1$) is a weak$^\ast$ closed subset of $B$, and so is weak$^{\ast}$ compact as well. Every element $f\in\ell^1(G)$ can be viewed as a bounded linear functional on $\ell^{\infty}(G)$, and so by (6), there is a net $(f_\alpha)$ of means on $\ell^\infty(G)$ such that for every $g\in G$ and $h\in\ell^\infty(G)$,
\begin{equation}
   \lim_{\alpha}\langle (^gf_\alpha-f_\alpha), h \rangle=0.   \tag{+} \label{star}
\end{equation}
(As they say, the net $(f_\alpha)$ weak$^\ast$ converges to invariance.)
Let $f$ be a weak$^\ast$ cluster point of this net, that is, a limit of a convergent subnet (which exists by weak$^\ast$ compactness).
By (\ref{star}), $f$ is invariant.
\smallskip

(1)$\Rightarrow$(5):
We need a version of classical Hall's matching theorem (for a proof, see e.g. \cite{bollobas}, Corollary III.3.11, or below).

\begin{theorem}[Hall's $(2,1)$-matching theorem]
Let $\Gamma=(V,E)=(A,B,E)$ be a bipartite graph, where $V$ denotes vertices, $E$ denotes edges, $V=A\sqcup B$. Assume the degree of every vertex in $A$ is finite. Suppose further that for every finite $X\subseteq A$, $\vert \Gamma(X)\vert\geq 2\vert X \vert$, 
where $\Gamma(X)$ denotes the set of edges having a vertex in $X$. Then there are two injections $i$ and $j$ with domains equal to $A$,
disjoint images in $B$, and such that $(a,i(a)),(a,j(a))\in E$, whenever $a\in A$.
\end{theorem}

Suppose that the F\o lner's condition fails. Fix a finite $F\subseteq G$, $\e>0$ so that for every finite $\Phi\subseteq G$ there is $g\in F$
such that
\[  \vert g\Phi\bigtriangleup\Phi\vert\geq \e \vert\Phi\vert. \]

Consider a graph  $\Gamma=(V,E)$  with $V=G\cup G$ and 
\[(g,h)\in E \Leftrightarrow \exists_{x\in F^k} \,\, h=xg,\]
where $k$ is large so that for every finite $X\subseteq A$ the condition $\vert \Gamma(X)\vert\geq 2\vert X \vert$ holds.

Apply Hall's theorem and get injections $i$ and $j$. For $s,t\in F^k$ define 
\[\Omega_{s,t}=\{g\in G\colon i(g)=sg \mbox{ and } j(g)=tg \}. \]
Then $\{s\Omega_{s,t}\colon s\in F^k\}\cup \{t\Omega_{s,t}\colon t\in F^k\}$ is a family of sets such that each two are either pairwise disjoint 
or equal. Note that 
\[\bigcup_{s\in F^k}s^{-1}\left(s\Omega_{s,t}\right)=G=\bigcup_{t\in F^k}t^{-1}\left(t\Omega_{s,t}\right).\]
This contradicts (1).
\end{proof}

The present elegant proof of the implication (1)$\Rightarrow$(5) is relatively recent, see Deuber, Simonovits and S\'os \cite{DSS}. In its present form, it only applies to discrete groups, and it would be interesting to know whether the idea can be made to work for locally compact groups as well, where the classical argument remains rather more complicated. 

For a detailed treatment of amenability and related topics in the same spirit, see the survey article \cite{dlHGCS}, containing in particular a proof of the $(2,1)$-matching theorem (\S 35). Here is a different argument.

\begin{proof}[Proof of Hall's $(2,1)$-matching theorem] 
It is enough to establish the result for finite graphs and use the standard compactness argument. (In the spirit of these notes: choose injections $i_{A^\prime}$, $j_{A^\prime}$ for every induced subgraph on vertices $A^\prime\sqcup \Gamma(A^\prime)$, where $A^\prime\subseteq A$ is finite. Choose a suitable ultrafilter $\mathcal U$ on the family of all finite subsets of $A$. The ultralimit $i(a)=\lim_{A^\prime\to{\mathcal U}}i_{}(a)$ is well-defined, simlarly for $j$, and the pair of injections $i,j$ is as desired.) 

We use induction on $n=\abs A$. For $n=1$ the result is obvious. Let now $\abs A=n+1$. Assume without loss in generality that $E$ is a minimal set of edges satisfying the assumptions of Theorem. It suffices now to verify that for every $a\in A$, $\abs{E(a)}=2$. 

Suppose towards a contradiction that there is an $a\in A$ with $E(a)\subseteq B$ containing at least three distinct points, $b,c,d$. The minimality of $E$ means none of the edges $(a,b)$, $(a,c)$ and $(a,d)$ can be removed, as witnessed by finite sets $X_b,X_c,X_d\subseteq A\setminus\{a\}$ with $\Gamma(X_b\cup\{a\})\setminus \{b\}$ containing $\leq 2\abs{X_b}+1$ points, and so on. This in particular implies $\abs{\Gamma(X_b\cup\{a\})}=2\abs{X_b}+2$, and similarly for $c$ and $d$. Since every set $X_z$, $z\in\{b,c,d\}$ must contain a point adjacent to a point in $\{b,c,d\}\setminus\{z\}$, at least one of these sets is a proper subset of $A\setminus\{a\}$. Fix a $z\in\{b,c,d\}$ with this property and denote $S=X_z\cup\{a\}$. One has $1\leq\abs S\leq n$ and $\abs{\Gamma(S)}=2\abs S$.

Let $\Gamma_1$ be the induced subgraph on vertices $S\sqcup\Gamma(S)$, and let $\Gamma_2$ be the induced subgraph on the remaining vertices of $\Gamma$, that is, $(A\setminus S)\sqcup (B\setminus \Gamma(S))$. Both $\Gamma_1$ and $\Gamma_2$ satisfy the assumptions of Theorem. For $\Gamma_1$ this is obvious: if $Y\subseteq S$, then $\Gamma_1(Y)=\Gamma(Y)$. For $\Gamma_2$, if we assume that $Z\subseteq A\setminus S$ is such that $\abs{\Gamma_2(Z)}<2\abs Z$, we get a contradiction:
\[\abs{\Gamma(S\sqcup Z)} =\abs{\Gamma(S)\cup \Gamma(Z)}=\abs{\Gamma(S)\sqcup\Gamma_2(Z)}=2\abs{S}+\abs{\Gamma_2(Z)}<2(\abs{S}+\abs Z).\]
Since the cardinality of $S$ and of $A\setminus S$ is less than $n+1$, the graphs $\Gamma_i$, $i=1,2$ admit $(2,1)$-matchings, say $i_1,j_1$ and $i_2,j_2$ respectively. The images of four mappings are all pairwise disjoint. A concatenation of $i$'s and $j$'s gives a $(2,1)$-matching of $\Gamma$, whose set of edges is strictly contained in $E$, contradicting the minimality of the latter. 
\end{proof}

\begin{theorem}
\label{th:amsof}
Every amenable group is sofic (hence hyperlinear).
\end{theorem}

\begin{proof}
Let $F\subseteq G$ be finite, and let $\e>0$. Choose a F\o lner set, $\Phi$, for the pair $(F,\e)$.

\begin{figure}[ht]
\begin{center}
\scalebox{0.25}[0.25]{\includegraphics{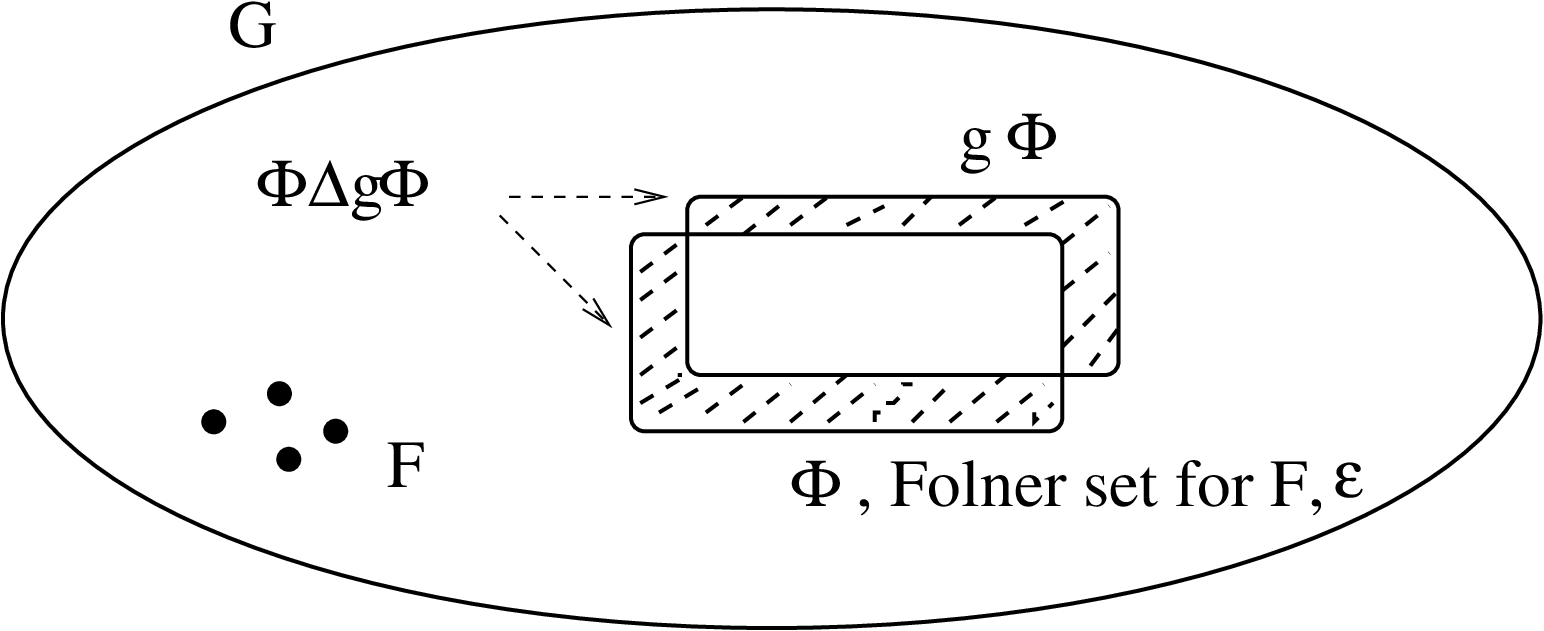}} 
\caption{A F\o lner set.}
\label{fig:folner}
\end{center}
\end{figure}

The map $x\mapsto gx$ is well-defined on a subset of $\Phi$ containing $>(1-\e)\abs{\Phi}$ points, and by extending it to a self-bijection of $\Phi$ one gets a $(F,2\e)$-homomorphism to the symmetric group $S_{\abs\Phi}$ satisfying condition (\ref{item:sofic3}) of Theorem \ref{th:soficcriterion}.
\end{proof}

The two results (Theorem \ref{th:lefsof} and Theorem \ref{th:amsof}) can be combined as follows. A group $G$ is {\em initially subamenable} (Gromov) if every finite subset $F\subseteq G$ admits an $(F,0)$-almost monomorphism into an amenable group $\Gamma$, that is, $F$ embeds into $\Gamma$ with the partial multiplication preserved. 

\begin{corollary}[Gromov]
Every initially subamenable group is sofic. 
\end{corollary}

As observed (independently) by Simon Thomas and Denis Osin, no finitely presented simple non-amenable group (such as Thompson's groups $V$ and $T$, for instance) is initially subamenable. (The argument is simple, and, as pointed out to me by the author of \cite{grigorchuk}, of the same kind as that used in the paper to show that the Grigorchuk groups are not finitely presented.)
Apparently, it remains unknown whether the three Thompson's groups are sofic. Examples of sofic groups which are not initially subamenable have been given by Cornulier \cite{cornulier}. (Earlier an example of a non-initially subamenable group which is hyperlinear was presented by Thom \cite{thom}.)

\section{Universal hyperlinear and sofic groups without ultraproducts}

Metric ultraproducts of groups $U(n)_2$ can be considered as universal hyperlinear groups, and those of groups $S_n$ as universal sofic groups. They are studied from this viewpoint in \cite{thomas}, where it is shown that if the Continuum Hypothesis fails, then there exist $2^{2^{\aleph_0}}$ pairwise-nonisomorphic metric ultraproducts of groups $S_n$ over the index set of integers. 

As seen from Theorems \ref{th:hypcriteria} and \ref{th:soficcriterion}, the definitions of hyperlinear and sofic groups can be restated in a form independent from ultraproducts. In view of this, one would expect the existence of ``canonical'' universal hyperlinear/sofic groups, independent of ultraproducts. Such a construction indeed exists.

Let $(G_\alpha,d_\alpha)$ be a family of groups equipped with bi-invariant metrics. As before (Equation (\ref{eq:ellinftytypesum})), form the $\ell^{\infty}$-type direct sum $\mathscr G=\oplus^{\ell^{\infty}}_{\alpha\in A}G_{\alpha}$ of the groups $G_\alpha$; when the diameters of $G_\alpha$ are uniformly bounded from above, $\mathscr G$ is just the direct product of the groups in question, equipped with the supremum distance. 

Now define the $c_0${\em -type sum} of groups $G_{\alpha}$ by letting
\[\oplus^{c_0}_{\alpha\in A}G_{\alpha} =
\{x\in\mathscr G\colon \lim_{\alpha}d(e_\alpha,x_{\alpha})=0\}.\]
In other words, $x\in \oplus^{c_0}_{\alpha\in A}G_{\alpha}$ if and only if
\[\forall\e>0,~~\{\alpha\in A\colon d(e_{\alpha},x_{\alpha})>\e \}\mbox{ is finite.}\]
It is easily seen that $\oplus^{c_0}_{\alpha\in A}G_{\alpha}$ forms a closed normal subgroup of $\mathscr G$, and so the quotient group $\oplus^{\ell^{\infty}}_{\alpha\in A}G_{\alpha}/\oplus^{c_0}_{\alpha\in A}G_{\alpha}$ is equipped with a complete bi-invariant metric. 

To simplify the notation, we will write
\[\oplus^{\ell^{\infty}/c_0}_{\alpha\in A}G_{\alpha}=\oplus^{\ell^{\infty}}_{\alpha\in A}G_{\alpha}/\oplus^{c_0}_{\alpha\in A}G_{\alpha},\]
and call the resulting complete metric group the {\em $\ell^{\infty}/c_0$-type product} of the groups $G_{\alpha}$, $\alpha\in A$.

\begin{example}
In the case where all $G_\alpha$ are equal to the additive group of the scalar field (e.g.\, $\R$ or $\C$), the resulting metric group is just the additive group of the well-known Banach space $\ell^{\infty}/c_0$, which is isometric to the space of continuous functions on the remainder $\beta\N\setminus\N$ of the Stone-\v Cech compactification of $\N$. This motivates our terminology and  notation.
\end{example}

\begin{theorem}
Let $G$ be a countable group.
\begin{enumerate}
\item $G$ is hyperlinear if and only if $G$ is isomorphic to a subgroup of $\oplus^{\ell^{\infty}/c_0}_{n\in\N}U(n)_2$.
\item $G$ is sofic if and only if $G$ is isomorpic to a subgroup of  $\oplus^{\ell^{\infty}/c_0}_{n\in\N}S_n$.
\end{enumerate}
\end{theorem}

\begin{proof}
We treat only the hyperlinear case. 
Write $G$ as $F_{\infty}/N$, where $N$ is a normal subgroup in the free group of countably many generators. Let $\pi\colon F_{\infty}\to G$ be the corresponding quotient homomorphism. For every $n$, denote $B_n$ the set of reduced words in $F_\infty$ of length $\leq n$ on the first $n$ generators of $F_\infty$, and let $\widetilde{B_n}=\pi(B_n)$ denote the image of $B_n$ in $G$. 

{\em Necessity.} 
Suppose $G$ is hyperlinear. For every $n$, fix a $(\widetilde{B_n},1/n^2)$-almost monomorphism $\tilde j_n$ to some unitary group $U(k_n)_2$ with images of every two distinct elements being at a distance at least $10^{-10}$. The composition $\pi\circ \tilde j_n$ is defined on the first $n$ generators of $F_{\infty}$. Extend it by the constant map $e$ over the rest of them. Denote by $j_n\colon F_{\infty}\to\prod_{n}U(n)_2$ the unique homomorphism on the free group assuming the given values on free generators. An induction on the word length using bi-invariance of the Hilbert-Schmidt metric shows that for all $x\in B_n$, one has $d(\tilde j_n(\pi(x)),j_n(x))<1/n$. In particular, for every $x\in B_n\cap N$, one has $d_{HS}(j_n(x),e)<1/n$, and if $x,y\in B_n$ and $x\neq y$, then $d(j_n(x),j_n(y))\geq 10^{-10}-2/n$. 

One can surely assume without loss of generality the unitary groups $U(k_n)$ to have distinct dimensions. Define a homomorphism $h\colon F_{\infty}\to \prod_nU(n)_2$ by setting
\[h(x)_n =\begin{cases} 
j_m(x),&\mbox{ if }n=k_m\mbox{ for some }m,\\
e,&\mbox{ otherwise.}
\end{cases}\]
This $h$ has two properties:
\begin{enumerate}
\item if $x\in N$, then $d(j_n(x),e)\to 0$ as $n\to\infty$, and therefore
$h(x)\in\oplus^{c_0}_{n\in\N}U(n)_2$,
\item if $x\neq y\mod N$, then $d(h(x),h(y))=\sup_nd(j_n(x),j_n(y))\geq 10^{-10}$.
\end{enumerate}
This means that the homomorphism $h$ taken modulo $\oplus^{c_0}_{n\in\N}U(n)_2$ factors through $N$ to determine a group monomorphism from $G$ into $\oplus^{\ell^{\infty}/c_0}_{n\in\N}U(n)_2$.

{\em Sufficiency.} Now suppose $G$ embeds as a subgroup into the $\ell^{\infty}/c_0$-type product of unitary groups.
Every non-principal ultrafilter $\mathcal U$ on the natural numbers gives rise to a quotient homomorphism from $\prod_nU(n)_2$ to the metric ultraproduct $\left(\prod_nU(n)_2\right)_{\mathcal U}$, and since  the normal subgroup $\oplus^{c_0}U(n)$ always maps to identity, we get a family of homomorphisms \[h_{\mathcal U}\colon\oplus^{\ell^{\infty}/c_0}_{n\in\N}U(n)_2\to \left(\prod_nU(n)_2\right)_{\mathcal U}.\]

Let $x\in \prod_nU(n)_2\setminus \oplus^{c_0}U(n)_2$. There are $\e>0$ and an infinite set $A\subseteq\N$ so that for all $n\in A$ one has $d_{HS,n}(x_n,e)>\e$. If now $\mathcal U$ is an ultrafilter on $\N$ containing the set $A$, then $d(h_{\mathcal U}(x),e)\geq \e$.

We have shown that the homomorphisms $h_{\mathcal U}$ separate points, and consequently the $\ell^\infty/c_0$ type sum of the unitary groups is isomorphic with a subgroup of the product of a family of hyperlinear groups. 
In order to conclude that $G$ is hyperlinear, it remains to notice that the class of hyperlinear groups is closed under passing to subgroups and direct products. The first is obvious, and for the second
embed the group $U(n)\times U(m)$ into $U(n+m)$ via block-diagonal matrices, make an adjustment for the distances, and use Theorem \ref{th:hypcriteria}.

The sofic case is completely similar.
\end{proof}

(The above result is very close in spirit to Prop. 11.1.4 in \cite{BO}.)

The above theorem can be generalized to groups $G$ of any cardinality, in which case the universal group will be the $\ell^{\infty}/c_0$-type product of the family of all groups $U(n)_2$ (respectively $S_n$), $n\in\N$, each one taken $\abs G$ times. 

In view of this result, the $\ell^\infty/c_0$-type products of metric groups deserve further attention, including from the viewpoint of the model theory of metric structures. 

In our view, it would also be interesting to know
what can be said about subgroups of $\ell^\infty/c_0$-type sums of other groups, especially the unitary groups $U(n)$, $n\in\N$, with the uniform operator metric.

\section{Sofic groups as defined by Gromov, and Gottschalk's conjecture}

Sofic groups were first defined by Gromov \cite{gromov99}, under the descriptive name of {\em groups with initially subamenable Cayley graphs} (the name ``sofic groups'' belongs to Benjy Weiss \cite{weiss}). In this Section we will finally state Gromov's original definition, and explain a motivation: the  Gottschalk's Surjunctivity Conjecture.

A directed graph $\Gamma$ is {\em edge-coloured} if there are a set $C$ of {\em colours} and a mapping $E(\Gamma)\to C$. We will also say that $\Gamma$ is {\em edge $C$-coloured}.

Here is a natural example how the edge-colouring comes about. 
Let $G$ be a finitely-generated group. Fix a finite symmetric set $V$ of generators of $G$ not containing the identity $e$. The {\em Cayley graph} of $G$ (defined by $V$) is a non-directed graph having elements of $G$ as vertices, with $(g,h)$ being adjacent if and only if $g^{-1}h\in V$, that is, there is an edge from $g$ to $h$ iff one can get to $h$ by multiplying $g$ with a generator $v\in V$ on the right. Since the generator $v$ associated to a given edge is unique, the Cayley graph is edge $V$-coloured.

The {\em word distance} in the Cayley graph of $G$ is the length of the shortest path between two vertices. It is a left-invariant metric. Notice that as finitely coloured graphs, every two closed balls of a given radius $N$ are naturally isomorphic to one another (by means of a uniquely defined left translation), which is why we will simply write $B_N$ without indicating the centre.

\begin{definition}[Gromov]
The Cayley graph of a finitely generated group $G$ is {\em initially subamenable} if for every natural $N$ and $\e>0$ there is a finite edge $V$-coloured graph $\Gamma$ with the property that for the fraction of at least $(1-\e)\abs \Gamma$ of vertices $x$ of $\Gamma$ the $N$-ball $B_N$ around $x$ is isomorphic, as an edge $V$-coloured graph, to the $N$-ball in $G$. 
\end{definition}

It can be seen directly that the above definition does not depend on the choice of a particular finite set of generators, but this will also follow from the next theorem.
The equivalence of the original Gromov's concept of a group whose Cayley graph is initially subamenable with most other definitions of soficity mentioned in these notes belongs to Elek and Szab\'o \cite{ES}. Notice that the restriction to finitely generated groups in Gromov's definition is inessential --- as it is should now be obvious to the reader, soficity is a local property in the sense that a group $G$ is sofic if and only if every finitely generated subgroup of $G$ is sofic.

\begin{theorem}
Let $G$ be a finitely generated group with a finite symmetric generating set $V$. Then $G$ is sofic if and only if the Cayley graph of $(G,V)$ is initially subamenable.
\label{th:graph}
\end{theorem}

\begin{proof} 
{\em Necessity.} Assume $G$ is sofic. Let $\e>0$ and $N\in\N$ be given. Choose a $(B_N,\e)$-almost monomorphism, $j$, to a permutation group $S_n$, with the property that the images of every two distinct elements of $B_N$ are at a distance $>(1-\e/\abs{B_N}^2)$ from each other. 
Define a directed graph $\Gamma$ whose vertices are integers $1,2,\ldots,n$ (i.e., elements of the set $[n]$ upon which $S_n$ acts by permutations), and $(m,k)$ is an edge coloured with a $v\in V$ if and only if $j(v)m=k$. If $v,u\in B_N$ and $v\neq u$, then for at least $(1-\e/\abs{B_N}^2)n$ vertices $m$ one has $v(m)\neq u(m)$. 
It follows that the map $v\mapsto v(m)$ is one-to-one on $B_N$, and consequently an isomorphism of $V$-coloured graphs, for at least $(1-\e)n$ vertices $m$.
\smallskip

$\Leftarrow$: In the presence of an edge-colouring, every element $w\in B_{N}$ determines a unique translation of $\Gamma$ that is well-defined at all but $<\e\abs\Gamma$ of its vertices (just follow, inside $\Gamma$, any particular string of colours leading up to $w$ in the original ball). This defines a $(B_N,\e)$-almost homomorphism into the permutation group on the vertices of $\Gamma$, which is uniformly $(1-\e)$-injective.
\end{proof}

The graphs $\Gamma$ as above can be considered as finite clones of $G$, grown artificially using some sort of genetic engineering.

Here is an open problem in topological dynamics that motivated Gromov to introduce sofic groups.
Let $G$ be a countable group, $A$ a finite set equipped with a discrete topology. The Tychonoff power $A^G$ is a Cantor space (i.e., a compact metrizable zero-dimensional space without isolated points), upon which $G$ acts by translations:
\[  (g\cdot x)(h)=x(g^{-1}h).\]
Equipped with this action of $G$ by homeomorphisms, $A^G$ is a {\em symbolic dynamical system,} or a {\em shift.}
An {\em isomorphism} between two compact $G$-spaces $X$ and $Y$ is a homeomorphism $f\colon X\to Y$ which commutes with the action of $G$:
\[f(g\cdot x) = g\cdot f(x)\mbox{ for all }g\in G,~x\in A^G.\]

\begin{conjecture}[Gottschalk's Surjunctivity Conjecture, 1973, \cite{gottschalk}]
For every countable group $G$ and every finite set $A$, the shift system $A^G$ contains no proper closed $G$-invariant subspace $X$ isomorphic to $A^G$ itself.
\label{co:gott}
\end{conjecture}

The Conjecture remains open as of time of writing these notes, and the following is the strongest result to date.

\begin{theorem}[Gromov \cite{gromov99}]
\label{gromov}
Gottschalk's Surjunctivity Conjecture holds for sofic groups.
\end{theorem}

\begin{proof}[Sketch of the proof]
Let $\Phi\colon A^G\to A^G$ be an endomorphism of $G$-spaces, that is, such a continuous mapping that $\Phi(gf)=g(\Phi(f))$ for every $g\in G$. Consider the mapping $A^G\to A$ which sends $f\mapsto \Phi(f)(e)$. Since 
$A$ is finite, the preimage of every $a\in A$ is a clopen set. Hence there is a finite $F'\subseteq G$ and 
$\Phi_0\colon A^{F'}\to A$ such that for every $g\in G$, 
\[ \Phi(f)(g)=\Phi(g^{-1}f(e))=\Phi_0(g^{-1}f\restriction F')=\Phi_0(f\restriction gF'). \]
Assume that $\Phi$ is injective, then there is an inverse $\Psi\colon {\mathrm{image}}\,(\Phi)\to A^G$. The map $\Psi$ is determined by a certain finite $F''\subseteq G$,
and by $\Psi_0\colon {\mathrm{image}}\,(\Phi)\restriction A^{F''}\to A$. That means that both $\Phi$ and its inverse are encoded locally.

Now assume in addition that $\Phi$ is not onto. Choose a finite symmetric subset $B_1\subseteq G$ which is 
big enough both to store complete information about $\Phi$ and its inverse, and so that the restriction of $\Phi(A^G)$ to $B_1$ is not onto. From now on, without loss in generality, we can replace $G$ with a subgroup generated by $B_1$.

Grow a finite $B_1$-coloured graph $\Gamma$ whose number of vertices we will denote $N=N(\e)=\abs{V(\Gamma)}$, which locally looks like $B_{5}$ around at least $(1-\e)N$ vertices.

Using the local representations $\Phi_0$ and $\Psi_0$, we construct maps $\widetilde{\Phi},\widetilde{\Psi}\colon A^{\Gamma}\to A^{\Gamma}$. It follows that the size of the image of $\widetilde{\Phi}$ is 
at least $\abs{A}^{(1-\e)N}$. 

Now choose in $\Gamma$ a maximal system of disjoint balls of radius $1$. The number of vertices in the union of those balls is at least $cN(1-\e)$ for some $0<c<1$ (which only depends on $G$ and $B_1$). 
It follows that $image(\widetilde{\Phi})$ has size at most $\abs{A}^{N-cN(1-\e)}(\abs{A}^{\abs{B_1}}-1)^{cN(1-\e)/\abs{B_1}}$. By combining the two observations, we get:
\[\abs{A}^{(1-\e)N}\leq \abs{A}^{N-cN(1-\e)}(\abs{A}^{\abs{B_1}}-1)^{cN(1-\e)/\abs{B_1}},\]
that is,
\[\abs{A}^{(1-\e)}\leq \abs{A}^{(1-c+c\e)}(\abs{A}^{\abs{B_r}}-1)^{c/\abs{B_1}}\]
for every $\e>0$. We get a contradiction by sending $\e\downarrow 0$.
\end{proof}

The above proof belongs to Benjy Weiss and is worked out in great detail in \cite{weiss}. The original proof of Gromov \cite{gromov99} was different.


It would be interesting to know whether Gromov's theorem can be extended to hyperlinear groups.

\section{Near actions and another criterion of soficity by Elek and Szab\'o}

Let $(X,\mu)$ be a measure space, where the measure $\mu$ is at least finitely additive. A {\em near-action} of a group $G$ on $(X,\mu)$ is an assignment to $g\in G$ of a measure-preserving mapping $\tau_g\colon X\to X$ defined $\mu$-a.e. in such a way that 
\[\tau_{gh}=\tau_g\circ\tau_h\]
in the common domain of definition of both sides. A near-action is {\em essentially free} if for every $g\neq e$ and for $\mu$-a.e. $x\in X$, $\tau_g x\neq x$.

\begin{theorem}[Elek and Szab\'o \cite{ES}]
A group $G$ is sofic if and only if it admits an essentially free near-action on a set $X$ equipped with a finitely-additive probability measure $\mu$ defined on the family ${\mathscr P}(X)$ of all subsets of $X$.
\label{th:essfree}
\end{theorem}

\begin{proof}[Sketch of the proof]
{\em Necessity} ($\Rightarrow$): Let $G$ be a sofic group. For every finite $F\subseteq G$ and each $k\in\N_+$ select a $(F,1/k)$-almost homomorphism $j_{(F,k)}$ with values in some permutation group $S_{n(F,k)}$, which is uniformly $(1-\e)$-injective. As usual, we think of $S_{n(F,k)}$ as the group of self-bijections of the set $[n(F,k)]=\{1,2,\ldots,n(F,k)\}$. Form a disjoint union
\[X=\bigsqcup_{F,k}\,[n(F,k)].\]
Let ${\mathcal U}$ be any ultrafilter on the directed set of all pairs $(F,k)$ containing every upper cone $\{(F,k)\colon F\supseteq F_0,~k\geq k_0\}$. 
For every $A\subseteq X$ the formula 
\[\mu(A)=\lim_{(F,k)\to{\mathcal U}}\frac{\abs{A\cap [n(F,k)]}}{n(F,k)}\]
defines a finitely-additive probability measure, $\mu$, on the power set of $X$.
Given $g\in G$, the rule
\[\tau_g(x) = j_{(F,k)}(x),\mbox{ if }x\in [n(F,k)]\mbox{ and }g\in F\]
defines $\mu$-a.e. a measure-preserving transformation of $X$. It is easy to verify that $\tau$ is an essentially free near-action of $G$.
\\[3mm]
{\em Sufficiency} ($\Leftarrow$): Here the proof follows rather closely the arguments used to establish the implications $(2)\Rightarrow(1)\Rightarrow(5)$ in Theorem \ref{amena}.
\end{proof}

The above criterion stresses yet again that soficity is a weaker version of amenability. To the best of authors' knowledge, no analogous criterion for hyperlinear groups is known yet.

\section{Discussion and further reading}
It is still hard to point to any obvious concrete candidates for examples of non-sofic or non-hyperlinear groups.

One class of groups rather allergic to amenability and its variations is formed by Kazhdan groups, or groups with property $(T)$.
Let $G$ be a group, and let $\pi\colon G\to U(\H)$ be a unitary representation. Then $\pi$ {\em admits almost invariant vectors} if for every finite
$F\subseteq G$ there are  $\e>0$ and $x\in \H$ such that $\norm{x}=1$, and for every $g\in F$, $\norm{x-\pi(g)(x)}<\e$.
A group $G$ has {\em Kazhdan's property}, or {\em property $(T)$}, if whenever a unitary representation $\pi$ of $G$ admits almost invariant vectors, it has a fixed non-zero vector. For an introduction into this vast subject, see \cite{BdlHV}. And is the simplest example of an ``allergy'' mentioned above.

\begin{theorem}\label{at}
If a group is amenable and has property $(T)$, then it is finite.
\end{theorem}

The proof follows at once from the definition combined with the following equivalent characterization of amenability:

{\em Reiter's condition (P2):} for every $F\subseteq G$ and $\e>0$, there is $f\in\ell^2(G)$ with 
$\norm{ f}_{\ell^2(G)}=1$ and such that for each $g\in F$, $\norm{ f-gf }_{\ell^2(G)}<\e$.

Here is a much more difficult result in the same vein:

\begin{theorem}[Kirchberg, Valette]
If a group with property (T) embeds into the group $U(R)$ (in particular, into its subgroup $[{\mathscr R}]$), then it is residually finite.
\end{theorem}

It is in view of such results that Ozawa asked whether every finitely generated Kazhdan group that is sofic is residually finite. A negative answer was announced by Thom \cite{thom}. Consequently, a hope to use property $(T)$ in order to construct non-hyperlinear groups is a bit diminished now, but surely not gone, as it remains in particular unknown whether finitely generated simple Kazhdan groups can be hyperlinear/sofic. 

The present notes have been organized so as to minimize an overlap with the survey \cite{pestov} by the first-named author. We recommend the survey as a useful complementary source for a number of topics which were not mentioned in the workshop lectures because of lack of time, including the origin and significance of the class of hyperlinear groups (Connes' Embedding Conjecture \cite{connes-injective,BO}), links of the present problematics with solving equations in groups, and more, as well as a longer bibliography and a number of (overwhelmingly still open) questions. Among interesting recent developments are  sofic measure-preserving equivalence relations \cite{EL} and a theory of entropy for measure-preserving actions of sofic groups \cite{bowen}.

\section*{Acknowledgements}
The first-named author thanks the Organizers of the 7$^{\mbox{\tiny th}}$ Appalachian set theory workshop, especially Ernest Schimmerling and Justin Moore, for their hospitality and patience. 
Thanks go to David Sherman for the illuminating historical remark at the end of Section \ref{s:hyp}, and to Peter Mester for correcting an oversight in the earlier version of the notes. The authors are grateful to a team of anomymous referees who have produced a most helpful report of an astonishing size (12 typed pages long).


\begin{thebibliography}{100}

\bibitem{BdlHV} M.B. Bekka, P. de la Harpe, and A. Valette,
{\em Kazhdan's Property $(T)$,} New Mathematical Monographs \textbf{11}, Cambridge University Press, 2008.

\bibitem{BS}
J.L. Bell and A.B. Slomson, {\em Models and Ultraproducts. An introduction,} Dover Publications, Inc., Mineola, NY, 2006 reprint of the 1974 3rd revised edition.

\bibitem{BYBHU} I. Ben Yaacov, A. Berenstein, C.W. Henson, and A. Usvyatsov, {\em Model Theory for Metric Structures,} in: Model theory with applications to algebra and analysis. Vol. 2, 315--427, 
London Math. Soc. Lecture Note Ser., \textbf{350}, Cambridge Univ. Press, Cambridge, 2008. 

\bibitem{bollobas}
B. Bollob\'as, {\em Modern Graph Theory,}
Graduate Texts in Mathematics, \textbf{184}, Springer-Verlag, New York, 1998. 

\bibitem{bowen}
L. Bowen, {\em
Measure conjugacy invariants for actions of countable sofic groups,}
J. Amer. Math. Soc. \textbf{23} (2010), 217--245. 

\bibitem{BO} N.P. Brown and N. Ozawa, {\em ${C}^*$-Algebras and Finite-Dimensional Approximations,} Graduate Studies in Mathematics \textbf{88},  American Mathematical Society, Providence, R.I., 2008.

\bibitem{CFP}
J.W. Cannon, W.J. Floyd, and W.R. Parry, \textit{Introductory notes on Richard Thompson's groups,} Enseign. Math. (2) \textbf{42} (1996), 215--256. 

\bibitem{connes-injective} A. Connes, 
{\it Classification of injective factors,}
Ann. of Math. {\bf 104} (1976), 73--115.

\bibitem{cornulier}
Y. Cornulier, {\em A sofic group away from amenable groups,}  Math. Ann. \textbf{350} (2011), 269--275.

\bibitem{dlHGCS}
P. de la Harpe, R.I. Grigorchuk, and T. Ceccherini-Silberstein,
{\em Amenability and paradoxical decompositions for pseudogroups and discrete metric spaces,} Tr. Mat. Inst. Steklova \textbf{224} (1999), Algebra. Topol. Differ. Uravn. i ikh Prilozh., 68--111 (in Russian); English translation in 
Proc. Steklov Inst. Math. \textbf{224} (1999), 57--97.

\bibitem{DSS}
W.A. Deuber, M. Simonovits, and V.T. S\'os, 
{\em A note on paradoxical metric spaces,}
Studia Sci. Math. Hungar. \textbf{30} (1995), 17--23. An annotated 2004 version is available at: \\
{\tt http://novell.math-inst.hu/$\sim$miki/walter07.pdf}

\bibitem{EL} G. Elek and G. Lippner, {\em
Sofic equivalence relations,} J. Funct. Anal. \textbf{258} (2010), 1692--1708. 

\bibitem{ES}
G. Elek and E. Szab\'o, {\em
Hyperlinearity, essentially free actions and $L\sp 2$-invariants. The sofic property,}
Math. Ann. 332 (2005), no. 2, 421--441. 


\bibitem{gottschalk} 
W. Gottschalk, {\em Some general dynamical notions,} in: Recent Advances in Topological Dynamics, Lecture Notes Math. \textbf{318}, Springer-Verlag, Berlin, 1973, pp. 120--125.

\bibitem{Gre1} F.P. Greenleaf,
{\it Invariant Means on Topological Groups,}
Van Nostrand Mathematical Studies {\bf 16},
Van Nostrand -- Reinhold Co., NY--Toronto--London--Melbourne, 1969.

\bibitem{grigorchuk}
R.I. Grigorchuk, {\em
Degrees of growth of finitely generated groups and the theory of invariant means,} Math. USSR-Izv. \textbf{25} (1985), no. 2, 259Ð300.

\bibitem{gromov99}
M. Gromov, {\em Endomorphisms of symbolic algebraic varieties,} J. Eur. Math. Soc. (JEMS) 1 (1999), no. 2, 109--197.

\bibitem{HR} E. Hewitt and K.A. Ross,
{\it Abstract Harmonic Analysis. Vol. 1 (2nd ed.),}
Springer--Verlag, NY a.o., 1979.



\bibitem{KM} A.S. Kechris and B.D. Miller,  
{\em Topics in orbit equivalence,}
Lecture Notes Math. \textbf{1852}, Springer-Verlag, Berlin, 2004.

\bibitem{KR} A.S. Kechris and C. Rosendal, 
\textit{Turbulence, amalgamation and generic automorphisms of 
homogeneous structures,} Proc. Lond. Math. Soc. (3) \textbf{94} (2007), 302--350. 

\bibitem{los} J. \L o\'s, {\em Quelques remarques, th\'eor\`emes et probl\`emes sur les classes d\'efinissables d'alg\`ebres,} in: Mathematical interpretation of formal systems, North-Holland Publishing Co., Amsterdam, 1955,  pp. 98--113.

\bibitem{ozawa}
N. Ozawa, {\em Hyperlinearity, sofic groups and applications to group theory,} handwritten note, 14 pp., August 2009, available at \\
{\tt http://www.ms.u-tokyo.ac.jp/$\sim$narutaka/NoteSofic.pdf} (accessed on April 24, 2012).

\bibitem{pestov}
V.G. Pestov, {\em Hyperlinear and sofic groups: a brief guide,} Bull. Symb. Logic \textbf{14} (2008), 449--480.

\bibitem{radulescu}
F. Radulescu, {\em The von Neumann algebra of the non-residually finite Baumslag group $\langle a,b\vert ab\sp 3a\sp {-1}=b\sp 2\rangle$ embeds into $R\sp \omega$,} in: Hot topics in operator theory, 173--185, Theta Ser. Adv. Math., 9, Theta, Bucharest, 2008 (prepublished as arXiv:math/0004172v3, 2000, 16 pp.)

\bibitem{sherman}
David Sherman, {\em Notes on automorphisms of ultrapowers of $II_1$ factors,}  Studia Math. \textbf{195} (2009), 201--217.

\bibitem{thom}
A. Thom,
{\em
Examples of hyperlinear groups without factorization property,}  Groups Geom. Dyn. \textbf{4} (2010), 195--208.

\bibitem{thomas}
S. Thomas, {\em On the number of universal sofic groups,}  Proc. Amer. Math. Soc. \textbf{138} (2010), 2585--2590.

\bibitem{VG}
A.M. Vershik and E.I. Gordon, {\em Groups that are locally embeddable in the class of finite groups,} St. Petersburg Math. J. \textbf{9} (1998), no. 1, 49--67.


\bibitem{wassermann} S. Wassermann, {\em On tensor products of certain group $C\sp{*} $-algebras,} J. Functional Analysis \textbf{23} (1976), 239--254. 

\bibitem{weiss}
B. Weiss, 
{\em Sofic groups and dynamical systems,}
Sankhy\=a Ser. A 62 (2000), no. 3, 350--359. \\
Available at: {\tt
http://202.54.54.147/search/62a3/eh06fnl.pdf}

\bibitem{wright}
F.B. Wright, {\em A reduction for algebras of finite type,}
Ann. of Math. (2) \textbf{60} (1954), 560--570. 
\end{thebibliography}
\end{document}